\documentclass[a4paper,12pt,reqno]{amsart}

\usepackage{stix2}
\usepackage{amsthm}
\usepackage{mathtools}
\usepackage[marginratio=1:1,scale=.7]{geometry}
\usepackage{tikz-cd}
\usepackage{tikz}
\usetikzlibrary{calc}
\usepackage[l2tabu,orthodox]{nag}
\usepackage{here}
\usepackage{graphicx}
\usepackage{enumitem}
   \setlist{leftmargin=0pt, itemindent=2.0em, labelsep=1em}
\usepackage[dvipsnames]{xcolor}
\usepackage{longtable}
\usepackage{adjustbox}
\usepackage{mathdots}
\usepackage{setspace}
\usepackage[breaklinks=true,colorlinks=true,linkcolor=Green,citecolor=MidnightBlue]{hyperref}
\usepackage{etoolbox}
\usepackage{aliascnt}

\usetikzlibrary{decorations.markings}
\usetikzlibrary{shapes.geometric}
\tikzset{every loop/.style={min distance=10mm,looseness=10}}
\tikzcdset{scale cd/.style={every label/.append style={scale=#1},
    cells={nodes={scale=#1}}}}
\usepackage{amsrefs}

\numberwithin{equation}{section}
\numberwithin{figure}{section}
\numberwithin{table}{section}
\allowdisplaybreaks

\setenumerate{label=(\arabic*),nosep}
\setitemize{nosep}
\newlist{clist}{enumerate}{1}
\setlist*[clist]{label=(\roman*), nosep}

\theoremstyle{plain}
\newtheorem{thm}{Theorem}[section]

\newaliascnt{lem}{thm}
\newtheorem{lem}[lem]{Lemma}
\aliascntresetthe{lem}

\newaliascnt{prp}{thm}
\newtheorem{prp}[prp]{Proposition}
\aliascntresetthe{prp}

\newaliascnt{cor}{thm}
\newtheorem{cor}[cor]{Corollary}
\aliascntresetthe{cor}

\newaliascnt{conj}{thm}

\aliascntresetthe{conj}

\theoremstyle{definition}

\newaliascnt{ass}{thm}

\aliascntresetthe{ass}

\newaliascnt{dfn}{thm}
\newtheorem{dfn}[dfn]{Definition}
\aliascntresetthe{dfn}

\newaliascnt{eg}{thm}
\newtheorem{eg}[eg]{Example}
\aliascntresetthe{eg}

\newaliascnt{rmk}{thm}
\newtheorem{rmk}[rmk]{Remark}
\aliascntresetthe{rmk}

\newtheorem*{ntn}{Notation}
\newtheorem*{ackn}{Acknowledgements}
\newtheorem*{org}{Organization}
\newtheorem*{aiuse}{Use of AI}

\usepackage{cleveref}
\crefname{thm}{Theorem}{Theorems}
\crefname{ass}{Assumption}{Assumptions}
\crefname{cor}{Corollary}{Corollaries}
\crefname{dfn}{Definition}{Definitions}
\crefname{lem}{Lemma}{Lemmas}
\crefname{ntn}{Notation}{Notations}
\crefname{prp}{Proposition}{Propositions}
\crefname{rmk}{Remark}{Remarks}
\crefname{eg}{Example}{Examples}
\crefname{section}{\S\!}{\S\S\!}
\crefname{equation}{equation}{equations}

\newcommand{\tit}{\textit}
\newcommand{\ol}{\overline}

\newcommand{\bbC}{\mathbb{C}}

\newcommand{\bbN}{\mathbb{N}}

\newcommand{\bbZ}{\mathbb{Z}}

\newcommand{\scL}{\mathscr{L}}

\newcommand{\calC}{\mathcal{C}}
\newcommand{\calI}{\mathcal{I}}

\newcommand{\calH}{\mathcal{H}}
\newcommand{\calM}{\mathcal{M}}

\DeclareMathOperator{\Ker}{Ker}

\DeclareMathOperator{\Ind}{Ind}
\DeclareMathOperator{\Irr}{Irr}
\DeclareMathOperator{\Stab}{Stab}
\DeclareMathOperator{\reg}{reg}

\DeclareMathOperator{\tr}{tr}
\DeclareMathOperator{\adj}{adj}
\DeclareMathOperator{\Hom}{Hom}
\DeclareMathOperator{\Res}{Res}
\DeclareMathOperator{\End}{End}

\makeatletter
\def\th@remark{%
  \thm@headfont{\bfseries}%
  \normalfont
  \thm@preskip\topsep \divide\thm@preskip\tw@
  \thm@postskip\thm@preskip
}
\makeatother

\begin{document}

\title[Brauer--Kuroda relations for ramified graph covers]
{Brauer--Kuroda relations for ramified graph covers}

\author{Kosuke Mizuno} 

\address{GRADUATE SCHOOL OF MATHEMATICS, NAGOYA UNIVERSITY, FURO-CHO, CHIKUSA-KU, \newline NAGOYA 464-8601, JAPAN}
\email{kosuke.mizuno.c1@math.nagoya-u.ac.jp}


\begin{abstract}
We establish graph-theoretic analogues of the Brauer--Kuroda relations for ramified Galois covers of finite graphs with arbitrary finite Galois groups.
These formulas relate the numbers of spanning trees of intermediate quotient graphs and recover the corresponding formulas for unramified covers.
Our proofs use \(h\)-functions defined on multiplicity spaces of complex representations of the Galois group.
\end{abstract}
\maketitle 
\tableofcontents 

\section{Introduction}

Let \(X\) be a finite connected graph, where a graph means a finite undirected multigraph, possibly with loops and multiple edges. Let \(\kappa(X)\) denote the number of spanning trees of \(X\), called the complexity of \(X\). This number equals the order of the degree-zero Picard group (or Jacobian group) of \(X\), which plays a role analogous to the ideal class group of a number field. In this analogy, \(\kappa(X)\) corresponds to the class number. For an unramified Galois cover with Galois group \((\bbZ/2\bbZ)^m\), Hammer, Mattman, Sands, and Vallières obtained a formula expressing the complexity of the covering graph in terms of the complexities of its intermediate double covers \cite[Theorem~3.6 and Remark~3.7]{HMSV24}. This formula is a graph-theoretic analogue of Kuroda's class number formula in algebraic number theory (For the number-theoretic formula, see \cite{Lem94}).

In \cite{Miz26}, we extend this formula to unramified \(G\)-covers for arbitrary finite groups \(G\), and establish graph-theoretic analogues of the Brauer--Kuroda relations. The purpose of this paper is to establish Brauer--Kuroda relations for ramified \(G\)-covers of finite graphs.

For a finite group \(G\), we consider \(G\)-covers \(Y/X\) in the sense of \cref{dfn:Galois}.
For a subgroup \(H\subseteq G\), let \(X_H\coloneq H\backslash Y\) be the quotient graph, so that \(Y/X_{H}\) is an \(H\)-cover.
For \(v\in V_X\), let \(m_v(Y/X)\) denote the ramification index at a vertex above \(v\). We define the ramification factor 
\[
f_H(Y/X)\coloneq\frac{\displaystyle\prod_{v\in V_X}m_v(Y/X)}{\displaystyle\prod_{x\in V_{X_H}}m_x(Y/X_H)}.
\]

Our first main result is a Kuroda-type formula.
Let
\[
\calH_G\coloneq\{\Ker\rho\mid\rho\in\Irr(G)\},
\]
where \(\Irr(G)\) denotes the set of irreducible complex representations of \(G\). 
Adjoining the empty set \(\emptyset\) to \(\calH_G\), we define
\[
\underline{\calH}_G \coloneq \calH_G\sqcup\{\emptyset\},
\]
which we order by inclusion.

\begin{thm}[\(=\) \cref{main3}]\label{main3n}
Let \(G\) be a finite group, and let \(\underline{\calH}_{G}\) be the partially ordered set defined above. 
Let \(\mu\colon \underline{\calH}_{G}\times \underline{\calH}_{G}\to \bbZ\) be the M\"obius function of \(\underline{\calH}_{G}\) (see \cref{dfn:mob}). 
Then, for any \(G\)-cover \(Y/X\), we have
\[
\kappa(Y)=
\frac{1}{|G|}
\left(\prod_{v\in V_X}m_v(Y/X)\right)
\prod_{H\in \calH_G}
\left( \frac{[G:H]\kappa(X_H)}{f_{H}(Y/X)}
\right)^{-\mu(\emptyset,H)}.
\]
\end{thm}

Our second main result is a Brauer--Kuroda-type formula.
We define
\[\calC_G\coloneq\{C\mid C\text{ is a cyclic subgroup of } G\}.
\]
For \(C\in\calC_G\), define
\[
a(C)\coloneq\frac{1}{[G:C]}\sum_{\substack{B\in\calC_G\\ C\subseteq B}}\mu_{\mathrm{cl}}([B:C]),
\]
where \(\mu_{\mathrm{cl}}\) denotes the classical number-theoretic M\"obius
function.

\begin{thm}[\(=\) \cref{main4}]\label{main4n}
Let \(G\) be a finite group, and let \(\calC_G\) and \(a(C)\) be as defined above. 
Then, for any \(G\)-cover \(Y/X\), we have
\[
\kappa(X)
=
\prod_{C\in\calC_G}
\left(
\frac{[G:C]\kappa(X_C)}
{f_C(Y/X)}
\right)^{a(C)}.
\]
\end{thm}

If \(Y/X\) is unramified, then \(m_v(Y/X)=1\) for every \(v\in V_X\), and \(f_H(Y/X)=1\) for every subgroup \(H\subseteq G\). Therefore, \cref{main3n,main4n} recover the spanning tree formulas in \cite[Theorem~4.3 and Theorem~4.11]{Miz26}. In particular, when \(G=(\bbZ/2\bbZ)^m\), \cref{main3n,main4n} recover the formula of Hammer, Mattman, Sands, and Valli\`eres (see \cref{eg:elementary-abelian-branched,eg:elementary-abelian-artin}).

Our proofs are based on \(h\)-functions defined from the three-term determinant arising in the theory of Ihara zeta functions and Artin--Ihara
\(L\)-functions.
Zakharov introduced Ihara zeta functions and Artin--Ihara \(L\)-functions for quotient graphs of groups with trivial edge stabilizers, and established two-term and three-term determinant formulas \cite{Zak21}.
For finite abelian group actions on graphs, Gambheera and Valli\`eres developed a theory of Ihara zeta and \(L\)-functions, established the additivity and induction properties of the \(L\)-functions, and applied this theory to Iwasawa theory for ramified graph covers \cite{GV26}.

Let \(G\) be a finite group and \(Y/X\) a \(G\)-cover. We define the \(h\)-functions in terms of the
multiplicity spaces
\[
\calM_\rho(Y)\coloneq\Hom_G(V_\rho,\bbC[V_Y])
\]
for finite-dimensional complex representations \(\rho\) of \(G\), where \(\bbC[V_Y]\) is the permutation representation
on the vertices of \(Y\).
For an irreducible representation \(\rho\), the dimension of \(\calM_\rho(Y)\) is the multiplicity of \(\rho\) in \(\bbC[V_Y]\), 
and the evaluation maps give
\[
\bbC[V_{Y}]
\cong
\bigoplus_{\rho\in \Irr(G)}
V_{\rho}\otimes\calM_{\rho}(Y)
\]
The adjacency and degree operators of \(Y\) commute with the
\(G\)-action, and hence they induce endomorphisms \(A_\rho\) and \(D_\rho\) of
\(\calM_\rho(Y)\).
We define the \(h\)-function
\[
h_{Y/X}(\rho,u)\coloneq
\det\bigl(I-A_\rho u+(D_\rho-I)u^2\bigr).
\]
The \(h\)-functions are multiplicative under direct sums, and Frobenius reciprocity gives compatibility with induction from arbitrary subgroups (For ramified covers, the inflation property can fail \cite[\S5.4]{GV26}).
At \(u=1\), \(h_{Y/X}(\rho,1)\) is the determinant of the Laplacian induced on \(\calM_{\rho}(Y)\). We obtain the spanning-tree factorization in \cref{main2} by factoring the characteristic polynomial of the Laplacian and applying the matrix-tree theorem. We then combine this factorization with the multiplicativity and induction properties of the \(h\)-functions to derive relations among the complexities of intermediate quotient graphs. We use this argument to derive the Brauer--Kuroda relations stated above.
\begin{org}
The paper is organized as follows.
In Section~2, we recall the basic definitions of finite graphs, ramified coverings, and \(G\)-covers. 
In Section~3, we describe ramified covers in terms of voltage assignments with inertia subgroups and prove a connectedness criterion for the derived graphs.
In Section~4, we introduce multiplicity spaces and \(h\)-functions for \(G\)-covers. We prove a ramified Hashimoto's formula, the spanning tree factorization, and the additivity and induction properties of the \(h\)-functions.
In Section~5, we prove the main theorems.
\end{org}

\begin{ackn}
The author thanks Yusuke Nakamura for his guidance and valuable comments, and is grateful to Ryosuke Murooka and Kentaro Tanaka for helpful discussions and suggestions. The author also thanks Taiga Adachi, Sohei Tateno, Takenori Kataoka, and Takuya Saito for their comments.
This work was partially supported by JST SPRING, Grant Number JPMJSP2125.
\end{ackn}
\begin{aiuse}
The author used ChatGPT (OpenAI) for mathematical discussions, computations, and improvement of the presentation. 
\end{aiuse}
\begin{ntn}
Throughout this paper, we use the following notation.
\begin{enumerate}
\item For a finite set \(A\), we denote its cardinality by \(|A|\).
\item All representations are finite-dimensional complex representations.
\item For a finite group \(G\), we denote by \(\mathbf{1}_G\) the
one-dimensional trivial representation of \(G\).
\item We denote by \([t^n]f(t)\) the coefficient of \(t^n\) in a polynomial \(f(t)\in\bbC[t]\), where \(n\) is a non-negative integer.
\end{enumerate}
\end{ntn}
\section{Graphs and ramified covers}\label{graph}
In this section, we recall the basic definitions and notation concerning graphs and ramified covers of graphs.
\subsection{Graphs and graph morphisms}\label{dfn of graphs}
We begin by recalling the basic definitions and notation for finite graphs and graph morphisms. We follow Serre's convention (see \cite{Sun13}).
\begin{dfn}[cf.~\cite{Ser77,Sun13}]
A \tit{graph} \(X=(V_X,E_X,o,t,\iota)\) consists of two sets \(V_X\) and \(E_X\), two maps
\(o,t\colon E_X\to V_X\), and a fixed-point-free involution \(\iota\colon E_X\to E_X\)
such that \(o(\iota(e))=t(e)\) for all \(e\in E_X\).
\end{dfn}
An element of \(V_X\) is called a \tit{vertex}, and an element of \(E_X\) is called an \tit{edge}. The maps \(o\) and \(t\) are called the \tit{origin map} and \tit{terminus map}, respectively. We usually write \(\ol{e}\coloneq \iota(e)\) for \(e \in E_X\). For \(v\in V_X\), we define 
\[
E_{X}^{v}\coloneq \{\,e\in E_X \mid o(e)=v\,\},\qquad \deg_{X}(v)\coloneq |E_{X}^{v}|.
\]  
The cardinality \(\deg_{X}(v)\) is called the \tit{degree} of \(v\). 
The \tit{Euler characteristic} \(\chi(X)\) of a finite graph \(X\) is defined by 
\[
\chi(X)\coloneq |V_X|-\frac{|E_X|}{2}.
\]


For a subset \(S\subseteq E_X\), set
\(\ol{S}\coloneq\{\,\ol{e}\mid e\in S\,\}.\)
An \tit{orientation} of \(X\) is a subset \(S\subseteq E_X\) such that
\(E_X=S\sqcup\ol{S}.\)


A \textit{path} in \(X\) is a finite sequence
\(p=(e_1,e_2,\dots ,e_n)\) 
of edges in \(E_{X}\) such that
\(t(e_i)=o(e_{i+1})\)
for \(i\in \{1,\dots, n-1\}\).
We write
\(o(p)\coloneq o(e_1)\)
and
\(t(p)\coloneq t(e_n).\)
A path \(p\) is \textit{closed} if \(o(p)=t(p)\). 
If \(o(p)=t(p)=v\), we say that \(p\) is a closed path based at \(v\).
For a path \(p=(e_1,\dots, e_n)\), its reverse path is defined by
\(
\overline{p}\coloneq (\overline{e_n},\dots,\overline{e_1}).
\)
We also regard each vertex \(v\in V_X\) as a path of length zero, called
the \textit{trivial path} at \(v\), and put
\(o(v)=t(v)=v\)
and
\(\ol{v}\coloneq v.\)

We next define three matrices associated with a finite graph.
\begin{dfn}
Let \(X\) be a finite graph.

\begin{enumerate}
\item The \tit{degree matrix} \(D_X\) of \(X\) is the diagonal matrix indexed by \(V_X\) with
\[
(D_X)_{v,v}\coloneq \deg_{X}(v).
\]
\item The \tit{adjacency matrix} \(A_X\) of \(X\) is the symmetric matrix indexed by \(V_X\) with
\[
(A_X)_{u,v}\coloneq \bigl|\{\, e\in E_X \mid o(e)=u,\ t(e)=v \,\}\bigr|.
\]\item The \tit{Laplacian matrix} \(L_X\) of \(X\) is defined by
\[
L_X\coloneq D_X-A_X.
\]
\end{enumerate}
\end{dfn}

We denote by \(\kappa(X)\) the number of spanning trees of a finite connected graph \(X\).

\begin{thm}[Matrix-tree theorem, cf.~\cite{Big93}]\label{thm:mattree}
Let \(X\) be a finite connected graph. Then the following hold.
\begin{enumerate}[label=\((\arabic*)\)]
\item \(\det(L_X(v))=\kappa(X)\) for every \(v\in V_{X}\), where \(L_X(v)\) denotes the matrix obtained from \(L_X\) by deleting the row and column indexed by \(v\),
\item \([t]\det(tI-L_X)=(-1)^{|V_X|-1}|V_X|\,\kappa(X)\),
\item \(\adj(L_{X})=\kappa(X)\ J\), where \(J\) is the \(|V_{X}|\times |V_{X}|\) matrix all of whose entries are \(1\).
\end{enumerate}
\end{thm}
We also recall the notion of a morphism between graphs.
\begin{dfn}
Let \(Y=(V_Y, E_Y)\) and \(X=(V_X, E_X)\) be graphs. A \tit{morphism} \(f=(f_V,f_E)\colon Y\to X\) of graphs consists of maps \(f_V\colon  V_Y\rightarrow V_X\) and \(f_E\colon E_Y\rightarrow E_X\) satisfying the following conditions, for all \(e\in E_{Y}\),
 \begin{enumerate}
	\item \(f_V(o(e))=o(f_E(e))\),
	\item \(f_V(t(e))=t(f_E(e))\),
	\item \(\ol{f_E(e)}=f_E(\ol{e})\).
\end{enumerate}
A morphism \(f\) is an \tit{isomorphism} if \(f_V\) and \(f_E\) are bijective. In this case, \(Y\) and \(X\) are said to be \tit{isomorphic}. When no confusion can arise, we write \(f\) instead of \(f_V\) and \(f_E\). 
\end{dfn}
\subsection{\texorpdfstring
  {Ramified covers and \(G\)-covers}
  {Ramified covers and G-covers}}
We recall the notion of a ramified cover of finite graphs.
We adopt the convention that the empty map
\(\emptyset\to\emptyset\)
is \(1\)-to-\(1\).
\begin{dfn}[cf.~{\cite[Definition~3.1]{GV24}}]\label{dfn:ram}
Let \(X\) and \(Y\) be finite connected graphs. A morphism
\(\pi=(\pi_V,\pi_E)\colon Y\to X\) is called a \tit{cover} if the following
conditions are satisfied.
\begin{enumerate}
\item The maps
\(\pi_V\colon V_Y\to V_X\)
and
\(\pi_E\colon E_Y\to E_X\)
are surjective.
\item For each \(w\in V_Y\), there exists a positive integer
\(m_w(Y/X)\) such that 
\(\pi_{E}\colon E_{Y}^{w}\to E_{X}^{\pi_{V}(w)}\)
is \(m_{w}(Y/X)\)-to-\(1\).
\end{enumerate}

The integer \(m_w(Y/X)\) is called the \tit{ramification index} of
\(\pi\) at \(w\).
The cover \(\pi\) is called \tit{unramified} if
\(m_w(Y/X)=1\) for every \(w\in V_Y\), and \tit{ramified} otherwise.
\end{dfn}
\begin{rmk}
Our convention for the empty map \(\emptyset\to\emptyset\) ensures that the ramification index is uniquely determined when \(E_X^{\pi_V(w)}=\emptyset\).
This occurs only when both \(Y\) and \(X\) consist of a single vertex and have no edges.
\end{rmk}
\begin{dfn}[cf.~{\cite[Definition~2.5]{Kat25}}]\label{dfn:Galois}
	Let \(G\) be a finite group. A cover \(\pi\colon Y\to X\) is called a \tit{\(G\)-cover} if \(G\) acts faithfully on \(Y\) by graph automorphisms and the following conditions hold:
\begin{enumerate}
	\item Each fiber of \(\pi_{V}\) and \(\pi_{E}\) is \(G\)-stable,
	\item \(G\) acts transitively on each fiber of \(\pi_{V}\),
	\item \(G\) acts freely and transitively on each fiber of \(\pi_{E}\).
\end{enumerate}
\end{dfn}
For \(w\in V_Y\), let
\(\Stab_G(w)\coloneq\{g\in G\mid gw=w\}\)
denote the stabilizer of \(w\) in \(G\).
For \(w\in V_{Y}\), we have \(m_{w}(Y/X)=|\Stab_{G}(w)|\).
Since \(G\) acts transitively on each vertex fiber, the stabilizers of any two vertices in the same fiber are conjugate. 
Hence \(m_w(Y/X)\) is constant on each vertex fiber.
Therefore, for each \(v\in V_{X}\), we define  
\[
m_{v}(Y/X)\coloneq m_{w}(Y/X),
\]
where \(w\in\pi_{V}^{-1}(v)\). 
Since \(G\) acts transitively on every fiber of \(\pi_{V}\) and freely and transitively on every fiber of \(\pi_{E}\), 
we have
\[
|V_{Y}|=|G|\sum_{v\in V_{X}}\frac{1}{m_{v}(Y/X)},
\qquad
|E_{Y}|=|G||E_{X}|.
\]
Therefore
\begin{equation}\label{eq:euler-Y}
\chi(Y)=|G|\left(\sum_{v\in V_{X}}\frac{1}{m_{v}(Y/X)}-\frac{|E_{X}|}{2}\right).
\end{equation}
\begin{dfn}
Let \(\pi\colon Y\to X\) be a \(G\)-cover, and let \(H\leq G\).
The \tit{intermediate quotient graph} associated with \(H\) is defined by
\[
X_H\coloneq H\backslash Y.
\]
The natural quotient map \(Y\to X_H\) is an \(H\)-cover.
(cf.~\cite[Proposition~3.3]{GV24})
\end{dfn}
\section{Voltage graphs and ramified covers}
In this section, we introduce voltage graphs with inertia subgroups
and derived graphs. We also prove a connectedness criterion
for derived graphs.
\subsection{Voltage assignments and derived graphs}
We recall the construction of derived graphs from voltage assignments equipped with inertia subgroups.
\begin{dfn}
Let \(X\) be a finite connected graph and \(G\) a finite group. A map \(\alpha\colon E_X\to G\) is called a \tit{voltage assignment} if
\[
\alpha(\bar e)=\alpha(e)^{-1}
\]
for all \(e\in E_X\).
\end{dfn}
For a path \(p=(e_1,\dots,e_n)\), define the \tit{voltage} of \(p\) by \(\alpha(p)\coloneq \alpha(e_1)\cdots\alpha(e_n)\). 
The voltage of a trivial path is defined to be \(1_{G}\).
\begin{dfn}[cf.~{\cite[\S4.1]{GV24}, \cite[Definition~2.9]{Kat25}}]\label{dfn:vol}
Let \(X\) be a finite connected graph, 
let \(G\) be a finite group, 
let \(\alpha\colon E_X\to G\) be a voltage assignment,
and \(\calI=\{I_v\}_{v\in V_X}\) a family of subgroups of \(G\). 
The quadruple
\[
(X,G,\alpha,\calI)
\]
is called a \tit{voltage graph}.
Its \tit{derived graph} \(X(\alpha,\calI)\) is defined by
\[
V_{X(\alpha,\calI)}\coloneq \bigsqcup_{v\in V_X}\bigl((G/I_v)\times\{v\}\bigr),
\qquad
E_{X(\alpha,\calI)}\coloneq G\times E_X.
\]
For \((\sigma,e)\in E_{X(\alpha,\calI)}\), we define
\[
o((\sigma,e))\coloneq (\sigma I_{o(e)},o(e)),\quad
t((\sigma,e))\coloneq (\sigma\alpha(e)I_{t(e)},t(e)),\quad
\ol{(\sigma,e)}\coloneq (\sigma\alpha(e),\ol{e}).
\]
If \(I_v=\{1_G\}\) for every \(v\in V_X\), we write \(X(\alpha)\) instead of \(X(\alpha,\calI)\).
\end{dfn}

The finite group \(G\) acts naturally on the graph \(Y\coloneq X(\alpha,\calI)\) by
\begin{equation}\tag{\(\spadesuit\)}\label{eq:G-action}
\tau\cdot(\sigma I_{v},v)\coloneq (\tau\sigma I_{v},v),\qquad \tau\cdot(\sigma ,e)\coloneq(\tau\sigma ,e).
\end{equation}
\begin{prp}[cf.~{\cite[Proposition~4.1]{GV24}}]
Let \((X,G,\alpha,\calI)\) be a voltage graph. 
Suppose that its derived graph \(Y\coloneq X(\alpha,\calI)\) is connected and 
that the action {\hypersetup{linkcolor=black}\eqref{eq:G-action}} is faithful.
Let 
\(\pi\colon Y\to X\)
be the morphism given by 
\[
\pi_V\colon
V_Y\to V_X;
\quad
(\sigma I_v,v)\mapsto v,
\quad
\pi_E\colon
E_Y\to E_X;
\quad
(\sigma,e)\mapsto e.
\]
Then the following hold.
\begin{enumerate}[label=\((\arabic*)\)]
\item The morphism \(\pi\) is a \(G\)-cover with respect to the action {\hypersetup{linkcolor=black}\eqref{eq:G-action}} of \(G\). 
\item For \(w=(\sigma I_{v}, v)\in V_{Y}\), \(\Stab_{G}(w)=\sigma I_{v}\sigma^{-1}\).
\item \(m_{v}(Y/X)=|I_{v}|\) for every \(v\in V_{X}\).
\end{enumerate}
\end{prp}
Thus, every voltage graph with connected derived graph and faithful
\(G\)-action gives rise to a \(G\)-cover.
Conversely, every \(G\)-cover can be described by a voltage assignment and a family of subgroups.
\begin{prp}[{\cite[Proposition~2.11]{Kat25}}]
Let \(\pi\colon Y\to X\) be a \(G\)-cover. Then there is a voltage assignment \(\alpha\colon E_{X}\to G\) and a family \(\calI=\{I_{v}\}_{v\in V_{X}}\) of subgroups of \(G\) such that \(Y\cong X(\alpha,\calI)\) as \(G\)-covers of \(X\).
\end{prp}
\subsection{Connectedness of derived graphs}
Although \(X\) is assumed to be connected, the derived graph
\(X(\alpha,\calI)\) need not be connected.
In the unramified case, a connectedness criterion for derived graphs
was given by Gross and Tucker
\cite[Theorem~2.5.1]{GT87}.

The action {\hypersetup{linkcolor=black}\eqref{eq:G-action}} of \(G\) on \(Y\) extends to paths.
More precisely, if \(P=(f_1,\dots,f_n)\) is a path in \(Y\) and \(\tau\in G\), we define
\[
\tau P
\coloneq
(\tau f_{1},\dots,\tau f_n).
\]
If \(P\) is the trivial path at \(x\in V_Y\), then
\(\tau P\) is defined to be the trivial path at \(\tau\cdot x\).
Thus, if \(P\) is a path from \(x\in V_Y\) to \(y\in V_Y\), 
then \(\tau P\) is a path from \(\tau\cdot x\) to \(\tau\cdot y\). 
Moreover, \(\tau\ol{P}=\ol{\tau P}.\)

\begin{prp}\label{prp:connectedness}
Let \((X,G,\alpha,\calI)\) be a voltage graph.
Fix a vertex \(v_0\in V_X\). For each \(v\in V_X\), choose a path \(p_v\) in \(X\) from \(v_0\) to \(v\),
with \(p_{v_0}\) equal to the trivial path at \(v_0\), 
and put \(g_v\coloneq \alpha(p_v)\).
Then \(X(\alpha,\calI)\) is connected if and only if
\[
G=
\left\langle
\alpha(C),\ g_v\sigma g_v^{-1}
\ \middle|\
C \text{ is a closed path at }v_0,\ 
v\in V_X,\ \sigma\in I_v
\right\rangle.
\]
In particular, \(X(\alpha)\) is connected if and only if  
\(
G=
\left\langle
\alpha(C)
\ \middle|\
C \text{ is a closed path at }v_0\right\rangle\).
\end{prp}
\begin{proof}
We set
\(
H\coloneq
\left\langle
\alpha(C),\ g_v\sigma g_v^{-1}
\ \middle|\
C \text{ is a closed path based at }v_0,\ 
v\in V_X,\ \sigma\in I_v
\right\rangle
\)
and \(Y\coloneq X(\alpha,\calI)\).
We define a subset \(R\) of \(G\)
\[
R\coloneq
\{\,r\in G\mid \text{ there exists a path from } (I_{v_0},v_0) \text{ to } (rI_{v_0},v_0) \text{ in } Y \, \}.
\]
Then \(R\) is a subgroup of \(G\).
Indeed, the trivial path at \((I_{v_0},v_0)\) shows that \(1_G\in R\).
If \(r,s\in R\), let \(P\) and \(Q\) be paths from
\((I_{v_0},v_0)\) to \((rI_{v_0},v_0)\) and
\((sI_{v_0},v_0)\), respectively. Then the concatenation of
\(P\) and \(rQ\) shows that \(rs\in R\). Similarly,
\(r^{-1}\overline{P}\) shows that \(r^{-1}\in R\).

Suppose first that \(H=G\).
If \(C\) is a closed path based at \(v_0\), 
then there is a lift of \(C\) starting at
\((I_{v_0},v_0)\) and ending at
\((\alpha(C)I_{v_0},v_0)\). Hence \(\alpha(C)\in R\).
Moreover, for \(v\in V_X\) and \(\sigma\in I_v\), there is a lift of
\(p_v\) from \((I_{v_0},v_0)\) to
\((g_vI_v,v)=(g_v\sigma I_v,v).\)
Starting from this vertex, there is a lift of \(\overline{p_v}\)
ending at
\((g_v\sigma g_v^{-1}I_{v_0},v_0).\)
Thus \(g_v\sigma g_v^{-1}\in R\), and hence \(H\subseteq R\).
Since \(H=G\), we have \(R=G\).

Let \((\tau I_v,v)\in V_Y\). Since
\(\tau g_v^{-1}\in R\), there is a path from
\((I_{v_0},v_0)\) to
\(
(\tau g_v^{-1}I_{v_0},v_0).
\)
Following this path by a lift of \(p_v\), we reach
\(
(\tau g_v^{-1}g_vI_v,v)=(\tau I_v,v).
\)
Therefore \(Y\) is connected.

Conversely, suppose that \(Y\) is connected, and let \(\tau\in G\).
Choose a path in \(Y\) from
\((I_{v_0},v_0)\) to \((\tau I_{v_0},v_0)\), with successive edges
\((\sigma_1,e_1),\ldots,(\sigma_n,e_n).\)
Put
\(v_j\coloneq t(e_j)\)
and
\(w_j\coloneq t((\sigma_j,e_j))\)
for
\(1\leq j\leq n\),
and \(w_0\coloneq(I_{v_0},v_0)\).

We claim that, for each \(j=0,\dots,n\), there exists \(h_j\in H\) such that
\[
w_j=(h_jg_{v_j}I_{v_j},v_j).
\]
This is clear for \(j=0\), since \(g_{v_0}=1_G\).
Suppose that the claim holds for some \(0\leq j<n\).
Since
\((\sigma_{j+1}I_{v_{j}},v_{j})=w_j=(h_jg_{v_j}I_{v_j},v_j),\)
there exists \(\sigma\in I_{v_j}\) such that
\(\sigma_{j+1}=h_jg_{v_j}\sigma.\)
Hence we have
\begin{align*}
w_{j+1}
&=
(h_jg_{v_j}\sigma\alpha(e_{j+1})I_{v_{j+1}},v_{j+1})\\
&=
\bigl(
h_j
(g_{v_j}\sigma g_{v_j}^{-1})
\alpha(p_{v_j}e_{j+1}\overline{p_{v_{j+1}}})
g_{v_{j+1}}I_{v_{j+1}},
v_{j+1}
\bigr).
\end{align*}
Both
\(g_{v_j}\sigma g_{v_j}^{-1}\)
and
\(\alpha(p_{v_j}e_{j+1}\overline{p_{v_{j+1}}})\)
belong to \(H\), so the claim follows by induction.

In particular,
\[
(\tau I_{v_0},v_0)=w_n=(h_ng_{v_0}I_{v_0},v_0)=(h_nI_{v_0},v_0)
\]
for some \(h_n\in H\).
Since \(I_{v_0}\subseteq H\), we obtain \(\tau\in H\).
As \(\tau\in G\) was arbitrary, \(G=H\).
\end{proof}
\begin{rmk}
The subgroup \(H\) in \cref{prp:connectedness} is independent of the
choice of the paths \(p_v\), and is independent of the base vertex
\(v_0\) up to conjugation in \(G\).
Indeed, replacing \(p_v\) by another path \(p_v'\) changes \(g_v\) by
left multiplication by \(\alpha(p_v'\ol{p_v})\), which is the voltage of a closed path
based at \(v_0\), and hence does not change the subgroup in \cref{prp:connectedness}.
If the base vertex is changed from \(v_0\) to \(v_1\), then
the subgroup \(H\) in \cref{prp:connectedness} is replaced by
\(\alpha(r)^{-1}H\alpha(r)\), where \(r\) is a path
from \(v_0\) to \(v_1\).
Thus the condition in \cref{prp:connectedness} is independent of these choices.
\end{rmk}
\begin{cor}
Let \((X,G,\alpha,\calI)\) be a voltage graph.
\begin{enumerate}[label=\((\arabic*)\)]
\item If \(X(\alpha)\) is connected, then \(X(\alpha,\calI)\) is connected (cf.~\cite[Lemma~4.4]{GV24}).
\item If there exists a totally ramified vertex \(v\in V_X\), that is, \(I_v=G\), then \(X(\alpha,\calI)\) is connected.
\item If \(G\) is abelian and \(\left\langle I_v\mid v\in V_X\right\rangle=G,\)
then \(X(\alpha,\calI)\) is connected.
\end{enumerate}
\end{cor}
\begin{proof}
These assertions follow immediately from \cref{prp:connectedness}.
\end{proof}
\section{\texorpdfstring{\(h\)-functions for ramified \(G\)-covers}{h-functions for ramified G-covers}}
In this section, we introduce \(h\)-functions associated with ramified \(G\)-covers and study representation-theoretic properties of \(h\)-functions.
\subsection{\texorpdfstring
  {Multiplicity spaces and \(h\)-functions}
  {Multiplicity spaces and h-functions}}

We introduce the multiplicity spaces associated with
representations of \(G\) and define the \(h\)-functions.


Let \(Y/X\) be a \(G\)-cover.
We set
\[
\bbC[V_{Y}]\coloneq \bigoplus_{w\in V_{Y}}\bbC w,
\]
which is the permutation representation associated with the action of \(G\) on \(V_{Y}\).
Since \(G\) acts on \(Y\) by graph automorphisms, the degree, adjacency, and Laplacian operators are \(G\)-equivariant,
that is, we have \(D_Y,A_Y,L_Y\in \End_G(\bbC[V_Y]).\)
\begin{dfn}
Let \(\rho\) be a representation of \(G\) with representation space \(V_\rho\).
Define the multiplicity space of \(\rho\) in \(\bbC[V_{Y}]\) by
\[
\calM_{\rho}(Y)\coloneq \Hom_{G}(V_{\rho},\bbC[V_{Y}]).
\]
We put \(d_{\rho}\coloneq\dim_{\bbC} V_{\rho}\) and \(s_{\rho}\coloneq \dim\calM_{\rho}(Y)\).
Every \(G\)-equivariant endomorphism \(T\in \End_{G}(\bbC[V_{Y}])\) induces an endomorphism 
\[
T_{\rho}\colon \calM_{\rho}(Y) \to \calM_{\rho}(Y);
\quad
F\mapsto T\circ F.
\]
The endomorphisms induced by \(D_Y\), \(A_Y\), and \(L_Y\) are denoted by \(D_\rho\), \(A_\rho\), and \(L_\rho\), respectively.
We define a \tit{\(h\)-function} associated with \(\rho\) by
\[
h_{Y/X}(\rho,u)\coloneq \det(I-A_{\rho}u+(D_{\rho}-I)u^2)\in \bbC[u].
\]
\end{dfn}
\begin{lem}\label{lem:dimension-multiplicity}
Let \(Y/X\) be a \(G\)-cover. For each \(v\in V_X\), choose
\(w_v\in\pi_V^{-1}(v)\).
Then we have the following.
\begin{enumerate}[label=\((\arabic*)\)]
\item For every representation \(\rho\) of \(G\), we obtain
\[
s_\rho=\sum_{v\in V_X}\dim V_\rho^{\Stab_G(w_v)}.
\]
In particular, we have \(s_{\mathbf{1}_G}=|V_X|\).
\item The number of vertices of \(Y\) is given by
\[
|V_Y|=\sum_{\rho\in\Irr(G)}d_\rho s_\rho.
\]
\end{enumerate}
\end{lem}
\begin{proof}
Since \(G\) acts transitively on each fiber \(\pi_V^{-1}(v)\) with
stabilizer \(I_v\coloneq \Stab_G(w_v)\), we have
\[
\bbC[V_Y]
\cong
\bigoplus_{v\in V_{X}}
\bbC[\pi^{-1}(v)]
\cong
\bigoplus_{v\in V_{X}}
\bbC[G/I_{v}]
\cong
\bigoplus_{v\in V_X}
\Ind_{I_v}^{G}(\mathbf{1}_{I_v}).
\]
as \(\bbC[G]\)-modules.
Hence, by Frobenius reciprocity,
\begin{align*}
s_\rho
&=
\dim\Hom_G(V_\rho,\bbC[V_Y])\\
&=
\sum_{v\in V_X}
\dim\Hom_G
\bigl(
V_\rho,\Ind_{I_v}^{G}(\mathbf{1}_{I_v})
\bigr)\\
&=
\sum_{v\in V_X}
\dim\Hom_{I_v}
\bigl(
V_\rho,\mathbf{1}_{I_v}
\bigr)\\
&=
\sum_{v\in V_X}\dim V_\rho^{I_v},
\end{align*}
which proves (1).
We remark that the right-hand side of (1) is independent of the choice of \(w_v\), since the stabilizers of two vertices in the same fiber are conjugate in \(G\).

Finally, the evaluation maps induce a \(G\)-equivariant isomorphism
\[
\bbC[V_Y] \cong \bigoplus_{\rho\in\Irr(G)} V_\rho\otimes\calM_\rho(Y).
\]
Taking dimensions yields
\[
|V_Y|=\sum_{\rho\in\Irr(G)}d_\rho s_\rho,
\]
which proves (2).
\end{proof}
\subsection{\texorpdfstring
  {\(h\)-functions for the trivial representation}
  {h-functions for the trivial representation}}
We compute the \(h\)-function associated with the trivial representation and prove a ramified Hashimoto's formula.
\begin{dfn}[\cite{Zak21}]
Let \(Y/X\) be a \(G\)-cover. The \tit{charge matrix} \(C_{Y/X}\) of \(Y/X\) is the diagonal matrix indexed by \(V_X\) defined by
\[
(C_{Y/X})_{v,v}\coloneq m_v(Y/X).
\]
\end{dfn}
\begin{prp}\label{prp:triv}
Let \(Y/X\) be a \(G\)-cover. Then we have
\(
D_{\mathbf{1}_{G}}=C_{Y/X}D_{X}
\)
and
\(
A_{\mathbf{1}_{G}}=C_{Y/X}A_{X}
\)
Hence we obtain
\[
h_{Y/X}(\mathbf{1}_G,u)=\det\left(I-C_{Y/X}A_Xu+(C_{Y/X}D_X-I)u^2\right).
\]
\end{prp}
\begin{proof}
We identify \(\Hom_G(V_{\mathbf{1}_G},\bbC[V_Y])\) with \(\bbC[V_Y]^G.\)
For each \(v\in V_X\), put
\[
y_v\coloneq\sum_{w\in\pi^{-1}(v)}w.
\]
Then \(\{y_v\}_{v\in V_X}\) is a basis of \(\bbC[V_Y]^G\).
With respect to the isomorphism
\[
\bbC[V_Y]^G\to\bbC[V_X];
\quad
y_{v}\mapsto v,
\]
we have the following commutative diagrams
\[
\begin{tikzcd}
\bbC[V_Y]^G
  \arrow[r, "\cong"]
  \arrow[d, "D_{\mathbf{1}_G}"']
&
\bbC[V_X]
  \arrow[d, "C_{Y/X}D_X"]
\\
\bbC[V_Y]^{G}
  \arrow[r, "\cong"]
&
\bbC[V_X]
\end{tikzcd}
\qquad
\begin{tikzcd}
\bbC[V_Y]^G
  \arrow[r, "\cong"]
  \arrow[d, "A_{\mathbf{1}_G}"']
&
\bbC[V_X]
  \arrow[d, "C_{Y/X}A_X"]
\\
\bbC[V_Y]^{G}
  \arrow[r, "\cong"]
&
\bbC[V_X]
\end{tikzcd}
\]
Hence we have
\[
D_{\mathbf{1}_G}=C_{Y/X}D_X,
\qquad
A_{\mathbf{1}_G}=C_{Y/X}A_X,
\]
as desired.
\end{proof}
The following theorem extends Hashimoto's formula \(h_{Y/X}'(\mathbf{1}_{G},1)=-2\chi(X)\kappa(X)\) (cf.~\cite{Has90}) from unramified to ramified \(G\)-covers.
\begin{thm}\label{thm:ramified-hashimoto}
For a \(G\)-cover \(Y/X\),
\[
h_{Y/X}'(\mathbf{1}_{G},1)=-\frac{2\chi(Y)\kappa(X)}{|G|}\prod_{v\in V_X}m_{v}(Y/X).
\]
In particular, if \(Y/X\) is unramified, then \(h_{Y/X}'(\mathbf{1}_{G},1)=-2\chi(X)\kappa(X).\)
\end{thm}
\begin{proof}
Put \(M(u)\coloneq I-C_{Y/X}A_Xu+(C_{Y/X}D_X-I)u^2.\)
Then, by \cref{prp:triv}, we have
\(
h_{Y/X}(\mathbf{1}_{G},u)=\det M(u).
\)
By Jacobi's formula, we have
\[
h_{Y/X}'(\mathbf{1}_{G},u)
=
\tr\Bigl(
\adj\bigl(M(u)\bigr)
\cdot
\bigl(M'(u)\bigr)
\Bigr).
\]
Since \(M(1)=C_{Y/X}L_{X}\) and \(M'(1)=-C_{Y/X}A_{X}+2(C_{Y/X}D_{X}-I)\), we obtain
\begin{align*}
h_{Y/X}'(\mathbf{1}_{G},1)
&=\tr\Bigl(
\adj(C_{Y/X}L_X)\bigl(-C_{Y/X}A_X+2(C_{Y/X}D_X-I)\bigr)
\Bigr) \\
&\overset{\mathclap{\text{\ref{thm:mattree}(3)}}}{=}
\tr\Bigl(
\bigl(\kappa(X)\cdot J\bigr)
\cdot \det(C_{Y/X})C_{Y/X}^{-1}
\cdot \bigl(-C_{Y/X}A_X+2(C_{Y/X}D_X-I)\bigr)
\Bigr)\\
&=\kappa(X)\prod_{v\in V_X}m_v(Y/X)
\tr\Bigl(J(L_X+D_X-2C_{Y/X}^{-1})\Bigr).
\end{align*}
Since
\(JL_X=0\),
\(\tr(JD_X)=|E_X|\),
and
\(\tr(JC_{Y/X}^{-1})=\sum_{v\in V_X}1/m_v(Y/X),\)
it follows that
\begin{align*}
h_{Y/X}'(\mathbf{1}_{G},1)
&=
-2\kappa(X)\prod_{v\in V_X}m_{v}(Y/X)
\left(\sum_{v\in V_X}\frac{1}{m_{v}(Y/X)}-\frac{|E_X|}{2}\right)\\
&\overset{\mathclap{\eqref{eq:euler-Y}}}{=}
-\frac{2\chi(Y)\kappa(X)}{|G|}\prod_{v\in V_X}m_{v}(Y/X).
\end{align*}

If \(Y/X\) is unramified, then \(\chi(Y)=|G|\chi(X)\)
and
\(m_v(Y/X)=1\) for every \(v\in V_X\).
Consequently, we obtain
\(h_{Y/X}'(\mathbf{1}_{G},1)=-2\chi(X)\kappa(X).\)
\end{proof}
\begin{eg}
Let \(G=\bbZ/3\bbZ\), written additively, and let \(X=K_3\) be the
complete graph on three vertices, with \(V_X=\{v_0,v_1,v_2\}.\)
Choose an orientation \(S=\{e_0,e_1,e_2\}\) of \(X\) such that
\((o(e_0),t(e_0))=(v_0,v_1),\)
\((o(e_1),t(e_1))=(v_1,v_2),\)
and
\((o(e_2),t(e_2))=(v_2,v_0).\)
Define a voltage assignment \(\alpha\colon E_X\to G\) by
\(
\alpha(e_0)=1
\)
and
\(
\alpha(e_1)=\alpha(e_2)=0.
\)
We put
\(I_{v_0}=G,\)
and
\(I_{v_1}=I_{v_2}=\{0\}.\)
Then we obtain the \(G\)-cover
\(\pi\colon Y=X(\alpha,\calI)\to X\)
shown in Figure~\ref{fig:derived-K3-Z3}.

\begin{figure}[htbp]
\centering
\begin{tikzpicture}
\node (A) at (-3,0){
\begin{tikzpicture}[
  baseline={([yshift=-0.5ex]current bounding box.center)}
]

\node[
  draw=none,
  minimum size=3cm,
  regular polygon,
  regular polygon sides=6
] (a) {};

\fill (a.center) circle[radius=1.5pt];

\foreach \x in {1,2,...,6}
  \fill (a.corner \x) circle[radius=1.5pt];

\foreach \y/\z in {1/2,3/4,5/6}{
  \path (a.center) edge (a.corner \y);
  \path (a.corner \y) edge (a.corner \z);
  \path (a.corner \z) edge (a.center);
}

\end{tikzpicture}
};
\node (C) at (0.5,0.5){\large\(\pi\)};
\draw[->,thick](-.75,0) to (1.75,0);
\node (B) at (3,0){
\begin{tikzpicture}[
  baseline={([yshift=-0.5ex]current bounding box.center)},
  scale=1.5
]

\node[
  draw=none,
  minimum size=2cm,
  regular polygon,
  regular polygon sides=3
] (b) {};

\foreach \x in {1,2,3}
  \fill (b.corner \x) circle[radius=1.5pt];

\foreach \y/\z in {1/2,2/3,3/1}
  \path (b.corner \y) edge (b.corner \z);

\end{tikzpicture}
};
\end{tikzpicture}

\caption{The derived graph
\(Y=X(\alpha,\calI)\) (left) and the base graph
\(X=K_3\) (right).}
\label{fig:derived-K3-Z3}
\end{figure}
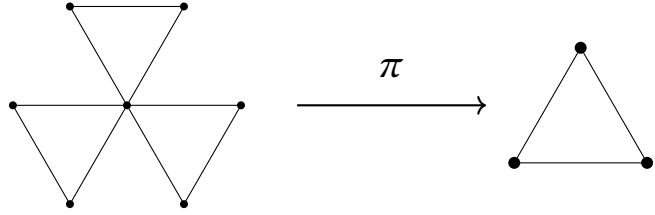
We compute \(h_{Y/X}(\rho,u)\) for
\(\rho\in\Irr(G)\).
Let
\(\omega\coloneq e^{2\pi i/3}.\)
Let
\(\Irr(G)=\{\rho_0,\rho_1,\rho_2\}\),
where
\(\rho_j(1)=\omega^j\).
First consider the trivial representation \(\mathbf{1}_{G}\).
We have
\[
A_{X}
=
\begin{pmatrix}
0&1&1\\
1&0&1\\
1&1&0
\end{pmatrix},
\qquad
D_{X}
=
\begin{pmatrix}
2&0&0\\
0&2&0\\
0&0&2
\end{pmatrix},
\qquad
C_{Y/X}
=
\begin{pmatrix}
3&0&0\\
0&1&0\\
0&0&1
\end{pmatrix}.
\]
By \cref{prp:triv}, it follows that
\begin{align*}
h_{Y/X}(\rho_0,u)
&=
\det(I-C_{Y/X}A_{X}u+(C_{Y/X}D_{X}-I)u^{2})\\
&=
1-6u^3+5u^6.
\end{align*}

Write
\(
V_Y
=
\{c\}
\sqcup
\{a_g\mid g\in G\}
\sqcup
\{b_g\mid g\in G\},
\)
where the vertex \(c\) lies above \(v_0\), the vertices \(a_g\) lie above
\(v_1\), and the vertices \(b_g\) lie above \(v_2\).
Since
\(I_{v_0}=G\)
and
\(I_{v_1}=I_{v_2}=\{0\},\)
we have
\(
\dim\calM_{\rho_j}(Y)
=
2
\)
for \(j\in \{1,2\}\).
Thus
\[
\{a_{j}\coloneq \sum_{g\in G}\omega^{-jg}a_g,b_{j}\coloneq \sum_{g\in G}\omega^{-jg}b_g\}
\]
is a basis of \(\calM_{\rho_j}(Y)\cong\{x\in\bbC[V_Y]\mid \tau x=\rho_j(\tau)x\text{ for every }\tau\in G\}\).

Since \(A_Ya_g=c+b_g\), \(A_Yb_g=c+a_g\), and
\(\deg_Y(a_g)=\deg_Y(b_g)=2\), with respect to the basis
\((a_j,b_j)\), we have
\[
A_{\rho_j}=
\begin{pmatrix}
0&1\\
1&0
\end{pmatrix},
\qquad
D_{\rho_j}=2I_2.
\]
Therefore we have
\begin{align*}
h_{Y/X}(\rho_j,u)
&=
\det\bigl(
I-A_{\rho_j}u+(D_{\rho_j}-I)u^2
\bigr)\\
&=
1+u^2+u^4.
\end{align*}
\end{eg}
\begin{eg}\label{eg:S3}
Let
\(
G=S_3
=
\langle s,r\mid s^2=r^3=1_G,\ srs=r^{-1}\rangle.
\) 
Let \(X=K_3\) be the complete graph on three vertices, with
\(
V_X=\{v_0,v_1,v_2\}.
\)
Let \(S=\{e_0,e_1,e_2\} \) be an orientation of \(X\), where
\((o(e_0),t(e_0))=(v_0,v_1)\), \((o(e_1),t(e_1))=(v_1,v_2)\), and \((o(e_2),t(e_2))=(v_2,v_0)\).
Define the trivial voltage assignment
\[
\alpha(e_0)=\alpha(e_1)=\alpha(e_2)=1_G
\]
and put
\(I_{v_0}=\langle s\rangle\), \(I_{v_1}=\langle r\rangle\), \(I_{v_2}=\{1_G\}\).
Then we have a \(G\)-cover \(\pi\colon Y=X(\alpha,\calI)\to X\) as in Figure \ref{fig:derived-K3-S3}.
The ramification indices are
\(
m_{v_0}(Y/X)=2,
\)
\(
m_{v_1}(Y/X)=3,
\)
and
\(
m_{v_2}(Y/X)=1.
\)
Since
\(
\langle I_{v_0},I_{v_1},I_{v_2}\rangle
=
\langle s,r\rangle
=
S_3,
\)
the derived graph \(Y\) is connected by \cref{prp:connectedness}. 

\begin{figure}[htbp]
\centering
\begin{tikzpicture}
\node (A) at (-3,0){
\begin{tikzpicture}
\node[
  draw=none,
  minimum size=3cm,
  regular polygon,
  regular polygon sides=11, 
  rotate=48
] (a) {};
\foreach \x in {1,2,...,11}{
  \fill (a.corner \x) circle[radius=1.5pt];
}
\foreach \y/\z in {1/3,1/2,2/3,1/7,1/11,11/7,5/3,5/4,4/3,5/7,5/6,6/7,9/3,9/10,10/3,9/7,9/8,8/7}{
  \path (a.corner \y) edge (a.corner \z);
}
\end{tikzpicture}
};
\node (C) at (0.5,0.5){\large\(\pi\)};
\draw[->,thick](-.75,0) to (1.75,0);
\node (B) at (3,0){
\begin{tikzpicture}[
]
\node[
  draw=none,
  minimum size=2cm,
  regular polygon,
  regular polygon sides=3
] (b) {};
\foreach \x in {1,2,3}{
  \fill (b.corner \x) circle[radius=1.5pt];
}
\foreach \y/\z in {1/2,2/3,3/1}{
  \path (b.corner \y) edge (b.corner \z);
}
\end{tikzpicture}
};
\end{tikzpicture}
\caption{The derived graph
\(Y=X(\alpha,\calI)\) (left) and the base graph
\(X=K_3\) (right).}
\label{fig:derived-K3-S3}
\end{figure}

The irreducible complex representations of \(S_3\) are the trivial
representation \(\mathbf{1}_{G}\), the sign representation \(\varepsilon\), and the
standard representation \(\rho\) of degree \(2\).

To describe bases of the nontrivial multiplicity spaces, identify
\[
\bbC[V_Y]
\cong
\bbC[G/\langle s\rangle]
\oplus\bbC[G/\langle r\rangle]
\oplus\bbC[G],
\]
where the summands correspond to the fibers above
\(v_0,v_1,v_2\), respectively.
For each subgroup \(H\subseteq G\), define the \(G\)-equivariant homomorphisms
\[
q_H\colon\bbC[G]\to\bbC[G/H];
\quad
g\mapsto gH,
\quad
\iota_H\colon\bbC[G/H]\to\bbC[G];
\quad
gH\mapsto\sum_{h\in H}gh.
\]
These homomorphisms satisfy
\[
q_H\circ\iota_H=|H|\,\mathrm{id}_{\bbC[G/H]}.
\]
For \(i=0,1\), the adjacency map from the fiber above \(v_i\)
to the fiber above \(v_2\) is \(\iota_{I_{v_i}}\), and the map
in the reverse direction is \(q_{I_{v_i}}\).

First consider the trivial representation \(\mathbf{1}_G\).
We have
\[
A_X
=
\begin{pmatrix}
0&1&1\\
1&0&1\\
1&1&0
\end{pmatrix},
\qquad
D_X
=
\begin{pmatrix}
2&0&0\\
0&2&0\\
0&0&2
\end{pmatrix},
\qquad
C_{Y/X}
=
\begin{pmatrix}
2&0&0\\
0&3&0\\
0&0&1
\end{pmatrix}.
\]
By \cref{prp:triv}, we obtain
\begin{align*}
h_{Y/X}(\mathbf{1}_G,u)
&=
\det\bigl(
I-C_{Y/X}A_{X}u+(C_{Y/X}D_{X}-I)u^2
\bigr)\\
&=
1-2u^2-12u^3-2u^4+15u^6.
\end{align*}

Consider the sign representation \(\varepsilon\).
Since
\(
V_\varepsilon^{\langle s\rangle}=0
\)
and
\(
V_\varepsilon^{\langle r\rangle}=V_\varepsilon,
\)
we have \(s_{\varepsilon}=2\) by \cref{lem:dimension-multiplicity} (1),
and hence
\[
\calM_{\varepsilon}(Y)=\Hom_{G}(V_{\varepsilon},\bbC[G/\langle r\rangle])\oplus\Hom_{G}(V_{\varepsilon},\bbC[G]).
\]
Let
\(
F_0\in
\Hom_G\bigl(
V_\rho,\bbC[G/\langle s\rangle]
\bigr)
\)
be nonzero, and set
\(
F_1\coloneq \iota_{\langle s\rangle}\circ F_0
\in \Hom_G(V_\rho,\bbC[G]).
\)
Since \(q_{\langle r\rangle}\circ\iota_{\langle r\rangle}
=3\,\mathrm{id}_{\bbC[G/\langle r\rangle]}\), the map \(F_2\) is nonzero.
Thus \((F_1,F_2)\) is a basis of \(\calM_\varepsilon(Y)\).
There is no \(\varepsilon\)-component in the fiber above \(v_0\), so
\[
A_\varepsilon F_1=F_2,
\qquad
A_\varepsilon F_2=3F_1.
\]
The degrees of the vertices above \(v_1\) and \(v_2\) are
\(6\) and \(2\), respectively. Hence, with respect to this basis \((F_{1},F_{2})\),
\[
A_\varepsilon
=
\begin{pmatrix}
0&3\\
1&0
\end{pmatrix},
\qquad
D_\varepsilon
=
\begin{pmatrix}
6&0\\
0&2
\end{pmatrix}.
\]
Therefore,
\begin{align*}
h_{Y/X}(\varepsilon,u)
&=
\det\bigl(
I-A_\varepsilon u+(D_\varepsilon-I)u^2
\bigr)\\
&=
1+3u^2+5u^4.
\end{align*}

Finally, consider the standard representation \(\rho\).
Since
\(\dim V_\rho^{\langle s\rangle}=1\)
and
\(V_\rho^{\langle r\rangle}=0\), we have \(s_{\rho}=3\) by \cref{lem:dimension-multiplicity} (1),
and hence 
\[
\calM_{\rho}(Y)=\Hom_{G}(V_{\rho},\bbC[G/\langle s\rangle])\oplus\Hom_{G}(V_{\rho},\bbC[G]).
\]
Let
\(F_0\in
\Hom_G\bigl(
V_\rho,\bbC[G/\langle s\rangle]
\bigr)\)
be nonzero, and set
\(F_1\coloneq \iota_{\langle s\rangle}\circ F_0
\in \Hom_G(V_\rho,\bbC[G]).\)
The map
\[
Q\colon
\Hom_G(V_\rho,\bbC[G])
\to
\Hom_G(V_\rho,\bbC[G/\langle s\rangle]);
\quad
F\mapsto q_{\langle s\rangle}\circ F.
\]
has a two-dimensional domain and a one-dimensional codomain.
Since \(Q(F_1)=2F_0\neq 0\), it is surjective and its kernel
is one-dimensional. Choose \(0\neq F_2\in\ker Q\).
Then \(F_1\) and \(F_2\) form a basis of
\(\Hom_G(V_\rho,\bbC[G])\), so \((F_0,F_1,F_2)\) is a basis
of \(\calM_\rho(Y)\).
There is no \(\rho\)-component in the fiber above \(v_1\), so
\[
A_\rho F_0=F_1,
\qquad
A_\rho F_1=2F_0,
\qquad
A_\rho F_2=0.
\]
The degrees of the vertices above \(v_0\) and \(v_2\) are
\(4\) and \(2\), respectively. Hence, with respect to this basis \((F_{0},F_{1},F_{2})\),
\[
A_\rho
=
\begin{pmatrix}
0&2&0\\
1&0&0\\
0&0&0
\end{pmatrix},
\qquad
D_\rho
=
\begin{pmatrix}
4&0&0\\
0&2&0\\
0&0&2
\end{pmatrix}.
\]
Thus
\begin{align*}
h_{Y/X}(\rho,u)
&=
\det\bigl(
I-A_\rho u+(D_\rho-I)u^2
\bigr)\\
&=
1+3u^2+5u^4+3u^6.
\end{align*}
\end{eg}
\subsection{Spanning-tree factorization}
In this section, we derive a factorization formula for the complexity
of a \(G\)-cover in terms of the \(h\)-functions associated with
irreducible representations of \(G\).

The following formula is obtained by a representation-theoretic argument.
A module-theoretic proof was given in \cite[Proposition~4.2]{Kat26}.
\begin{thm}\label{main2}
Let \(Y/X\) be a \(G\)-cover.
Then
\[
\kappa(Y)=\frac{\kappa(X)}{|G|}
\prod_{v\in V_X}m_v(Y/X)
\prod_{\rho\in \Irr(G)\setminus\{\mathbf{1}_{G}\}}h_{Y/X}(\rho,1)^{d_{\rho}}.
\]
\end{thm}
\begin{proof}
Let \(\reg_{G}\) be the regular representation of \(G\).
Fixing a decomposition
\[
V_{\reg_G}
\cong
\bigoplus_{\rho\in\Irr(G)}V_\rho^{\oplus d_\rho},
\]
we obtain the following commutative diagram
\[
\begin{tikzcd}[column sep=large]
\displaystyle
\bigoplus_{\rho\in\Irr(G)}
\Hom_G(V_\rho,\bbC[V_Y])^{\oplus d_\rho}
  \arrow[r, "\cong"]
  \arrow[d,
    "\bigoplus_{\rho\in\Irr(G)} L_\rho^{\oplus d_\rho}"]
&
\Hom_G(V_{\reg_G},\bbC[V_Y])
  \arrow[r, "\cong"]
  \arrow[d, "L_{\reg_G}"]
&
\bbC[V_Y]
  \arrow[d, "L_Y"]
\\
\displaystyle
\bigoplus_{\rho\in\Irr(G)}
\Hom_G(V_\rho,\bbC[V_Y])^{\oplus d_\rho}
  \arrow[r, "\cong"]
&
\Hom_G(V_{\reg_G},\bbC[V_Y])
  \arrow[r, "\cong"]
&
\bbC[V_Y].
\end{tikzcd}
\]
Hence we have
\begin{equation}\label{eq:LY-factorization}
\det(tI-L_{Y})=\prod_{\rho\in\Irr(G)}\det(tI-L_{\rho})^{d_{\rho}}
\end{equation}
We compare the coefficients of \(t\) on both sides.
By \cref{thm:mattree} (2), since \(Y\) is connected, we have
\begin{equation}\label{eq:LY-t-coefficient}
[t]\det(tI-L_{Y})=(-1)^{|V_{Y}|-1}|V_{Y}|\kappa(Y)
\end{equation}
For the trivial representation \(\mathbf{1}_{G}\), we have
\begin{align}
[t]\det(tI-L_{\mathbf{1}_{G}})
&\overset{\mathclap{\ref{prp:triv}}}{=}
[t]\det(tI-C_{Y/X}L_{X})\notag\\
&=
(-1)^{|V_{X}|-1}\sum_{v\in V_{X}}\det((C_{Y/X}L_{X})(v))\notag\\
&\overset{\mathclap{\ref{thm:mattree} (1)}}{=}
(-1)^{|V_{X}|-1}\kappa(X)\prod_{v\in V_{X}}m_{v}(Y/X)\sum_{v\in V_{X}}\frac{1}{m_{v}(Y/X)}.
\label{eq:trivial-t-coefficient}
\end{align}
We have
\begin{equation}\label{eq:nontrivial-constant}
[t^{0}]\prod_{\rho\in\Irr(G)\setminus\{\mathbf{1}_{G}\}}
\det(tI-L_{\rho})^{d_{\rho}}
=
(-1)^{\sum_{\rho\in\Irr(G)\setminus\{\mathbf{1}_{G}\}}d_{\rho}s_{\rho}}
\hspace{-4mm}
\prod_{\rho\in\Irr(G)\setminus\{\mathbf{1}_{G}\}}h_{Y/X}(\rho,1)^{d_{\rho}}
\end{equation}
Since we have \(|V_{Y}|=|V_{X}|+\sum_{\rho\in\Irr(G)\setminus\{\mathbf{1}_{G}\}}d_{\rho}s_{\rho}\) by \cref{lem:dimension-multiplicity},
it follows from 
\eqref{eq:trivial-t-coefficient},
\eqref{eq:nontrivial-constant}, and \([t^{0}]\det(tI-L_{\mathbf{1}_{G}})=0\) that
\begin{align}
&[t]\prod_{\rho\in\Irr(G)}\det(tI_{s_{\rho}}-L_{\rho})^{d_{\rho}}\notag\\
=&
\,\,(-1)^{|V_{Y}|-1}
\kappa(X)
\prod_{v\in V_{X}}m_{v}(Y/X)
\sum_{v\in V_{X}}\frac{1}{m_{v}(Y/X)}
\prod_{\rho\in\Irr(G)\setminus\{\mathbf{1}_{G}\}}h_{Y/X}(\rho,1)^{d_{\rho}}
\label{eq:t-coeff}
\end{align}
By \eqref{eq:LY-factorization}, \eqref{eq:LY-t-coefficient}, and \eqref{eq:t-coeff}, since \(|V_{Y}|=|G|\sum_{v\in V_{X}}1/m_{v}(Y/X)\),
\[
\kappa(Y)=\frac{\kappa(X)}{|G|}
\prod_{v\in V_X}m_v(Y/X)
\prod_{\rho\in \Irr(G)\setminus\{\mathbf{1}_{G}\}}h_{Y/X}(\rho,1)^{d_{\rho}}
\]
as desired.
\end{proof}
\begin{rmk}
\cref{main2} holds without assuming that
\(\chi(X)\neq 0\) or \(\chi(Y)\neq 0\).
In proofs based on Hashimoto's formula, a nonvanishing assumption
is needed when dividing by the Euler characteristic.
\end{rmk}
\subsection{\texorpdfstring
  {Artin formalism of \(h\)-functions}
  {Artin formalism of h-functions}}
We prove the Artin formalism for the \(h\)-functions by
establishing their additivity and induction properties for arbitrary finite groups.
In the finite abelian case, these properties were proved in \cite[Theorems~5.14 and 5.15]{GV26}.
\begin{prp}[Additivity]\label{prp:additivity-h-function}
Let \(\rho\) and \(\tau\) be representations of a finite group \(G\). Then
\[
h_{Y/X}(\rho\oplus\tau, u)=h_{Y/X}(\rho,u)h_{Y/X}(\tau,u).
\]
\end{prp}
\begin{proof}
For any \(T\in\End_{G}(\bbC[V_{Y}])\), the following diagram commutes
\[
\begin{tikzcd}
\calM_{\rho\oplus\tau}(Y)
  \arrow[r, "\cong"]
  \arrow[d, "T_{\rho\oplus\tau}"']
&
\calM_{\rho}(Y)\oplus\calM_{\tau}(Y)
  \arrow[d, "T_\rho\oplus T_\tau"]
\\
\calM_{\rho\oplus\tau}(Y)
  \arrow[r, "\cong"]
&
\calM_{\rho}(Y)\oplus\calM_{\tau}(Y).
\end{tikzcd}
\]
Applying this to \(T=A_Y\) and \(T=D_Y\), respectively, we obtain
\[
I-A_{\rho\oplus\tau}u
 +(D_{\rho\oplus\tau}-I)u^2\\
\cong
\bigl(I-A_\rho u+(D_\rho-I)u^2\bigr)
\oplus
\bigl(I-A_\tau u+(D_\tau-I)u^2\bigr).
\]
Taking determinants, we obtain
\[
h_{Y/X}(\rho\oplus\tau,u)
=
h_{Y/X}(\rho,u)h_{Y/X}(\tau,u),
\]
as desired.
\end{proof}
\begin{prp}[Induction property]\label{thm:induction-h-function}
Let \(Y/X\) be a \(G\)-cover, let \(H\) be a subgroup of \(G\), and let
\(\tau\) be a representation of \(H\). Then
\[
h_{Y/X}\bigl(\Ind_H^G(\tau),u\bigr)
=
h_{Y/X_H}(\tau,u).
\]
\end{prp}
\begin{proof}
For any \(T\in\End_{G}(\bbC[V_{Y}])\), Frobenius reciprocity gives the commutative diagram
\[
\begin{tikzcd}
  \calM_{\Ind^{G}_{H}\tau}(Y) \ar[r, "\cong"] \arrow[d, "T_{\Ind_{H}^{G}(\tau)}"'] & \calM_{\tau}(Y) \ar[d, "T_{\tau}"] \\
  \calM_{\Ind^{G}_{H}\tau}(Y) \ar[r, "\cong"] & \calM_{\tau}(Y).
\end{tikzcd}
\]
Applying this to \(T=A_Y\) and \(T=D_Y\), respectively, we obtain
\[
I-A_{\Ind_H^G(\tau)}u
 +(D_{\Ind_H^G(\tau)}-I)u^2
 \cong
I-A_\tau u+(D_\tau-I)u^2.
\]
Taking determinants, we have
\[
h_{Y/X}\bigl(\Ind_H^G(\tau),u\bigr)
=
h_{Y/X_H}(\tau,u),
\]
as desired.
\end{proof}
We next use the Artin formalism to derive a formula relating the
complexities of \(X_H\) and \(X\) in terms of the \(h\)-functions.
To state the formula, we introduce the following ramification factor.
\begin{dfn}\label{dfn:ramification-factor}
Let \(Y/X\) be a \(G\)-cover. For a subgroup \(H\subseteq G\), define the \tit{ramification factor} by
\[
f_H(Y/X)
\coloneq
\frac{
\displaystyle\prod_{v\in V_X}m_v(Y/X)
}{
\displaystyle\prod_{x\in V_{X_H}}m_x(Y/X_H)
}.
\]
\end{dfn}
\begin{lem}\label{lem:main}
Let \(Y/X\) be a \(G\)-cover, and let \(H\) be a subgroup of \(G\).
For each \(\rho\in\Irr(G)\setminus\{\mathbf{1}_{G}\}\), let
\(a_\rho^H\) be the non-negative integer determined by
\[
\Ind^{G}_{H}(\mathbf{1}_{H})=\mathbf{1}_{G}\oplus \bigoplus_{\rho\in \Irr(G)\setminus\{\mathbf{1}_{G}\}}\rho^{\oplus a_{\rho}^{H}}
\]
Then we have
\[
[G:H]\kappa(X_H)
=
f_{H}(Y/X)
\kappa(X)
\hspace{-4mm}
\prod_{\rho\in\Irr(G)\setminus\{\mathbf{1}_{G}\}}
h_{Y/X}(\rho,1)^{a_\rho^H}.
\]
\end{lem}
\begin{proof}
Applying \cref{main2} to the \(G\)-cover
\(Y\to X\) and the \(H\)-cover \(Y\to X_H\), respectively, gives
\[
\kappa(Y)
=
\frac{\kappa(X)}{|G|}
\prod_{v\in V_{X}}m_{v}(Y/X)\prod_{\rho\in \Irr(G)\setminus \{\mathbf{1}_{G}\}}h_{Y/X}(\rho,1)^{d_{\rho}},
\]
\[
\kappa(Y)
=
\frac{\kappa(X_H)}{|H|}
\prod_{x\in V_{X_{H}}}m_{x}(Y/X_{H})\prod_{\tau\in \Irr(H)\setminus \{\mathbf{1}_{H}\}}h_{Y/X_{H}}(\tau,1)^{d_{\tau}}.
\]
Equating these two expressions for \(\kappa(Y)\) and using the definition of \(f_H(Y/X)\), we obtain
\begin{equation}\label{eq:c}
[G:H]\kappa(X_{H})=f_{H}(Y/X)\kappa(X)
\frac{
\displaystyle
\prod_{\rho\in\Irr(G)\setminus\{\mathbf{1}_G\}}
h_{Y/X}(\rho,1)^{d_\rho}
}{
\displaystyle
\prod_{\tau\in\Irr(H)\setminus\{\mathbf{1}_H\}}
h_{Y/X_H}(\tau,1)^{d_\tau}
}.
\end{equation}
Since \(\Ind_H^G(\reg_H)\cong\reg_G\), it follows from
\cref{prp:additivity-h-function,thm:induction-h-function} that
\begin{equation}\label{eq:a}
\prod_{\tau\in \Irr(H)}h_{Y/X_{H}}(\tau,u)^{d_{\tau}}=\prod_{\rho\in \Irr(G)}h_{Y/X}(\rho,u)^{d_{\rho}}.
\end{equation}
Since \(\Ind^{G}_{H}(\mathbf{1}_{H})=\mathbf{1}_{G}\oplus \bigoplus_{\rho\in \Irr(G)\setminus\{\mathbf{1}_{G}\}}\rho^{\oplus a_{\rho}^{H}}\),
the induction property and the multiplicativity of \(h\)-functions give
\begin{equation}\label{eq:b}
h_{Y/X_{H}}(\mathbf{1}_{H},u)=h_{Y/X}(\mathbf{1}_{G},u)\prod_{\rho\in \Irr(G)\setminus\{\mathbf{1}_{G}\}}h_{Y/X}(\rho,u)^{a_{\rho}^{H}}.
\end{equation}
By \eqref{eq:a} and \eqref{eq:b},
\begin{align*}
\frac{
\displaystyle
\prod_{\rho\in\Irr(G)\setminus\{\mathbf{1}_G\}}
h_{Y/X}(\rho,u)^{d_\rho}
}{
\displaystyle
\prod_{\tau\in\Irr(H)\setminus\{\mathbf{1}_H\}}
h_{Y/X_H}(\tau,u)^{d_\tau}
}
&\overset{\eqref{eq:a}}{=}
\frac{
h_{Y/X_{H}}(\mathbf{1}_{H},u)
}{
h_{Y/X}(\mathbf{1}_{G},u)
}\\
&\overset{\eqref{eq:b}}{=}
\prod_{\rho\in \Irr(G)\setminus\{\mathbf{1}_{G}\}}h_{Y/X}(\rho,u)^{a_{\rho}^{H}}.
\end{align*}
Therefore, by \eqref{eq:c} we conclude that
\[
[G:H]\kappa(X_H)
=
f_{H}(Y/X)
\kappa(X)
\hspace{-4mm}
\prod_{\rho\in\Irr(G)\setminus\{\mathbf{1}_{G}\}}
h_{Y/X}(\rho,1)^{a_\rho^H},
\]
as desired.
\end{proof}
\begin{cor}\label{cor:normal-intermediate-complexity}
Let \(Y/X\) be a \(G\)-cover, and let \(H\) be a normal subgroup of \(G\).
Then
\[
[G:H]\kappa(X_H)
=
f_H(Y/X)\kappa(X)
\prod_{\substack{
\rho\in\Irr(G)\setminus\{\mathbf{1}_G\}\\
H\subseteq\Ker\rho
}}
h_{Y/X}(\rho,1)^{d_\rho}.
\]
\end{cor}
\begin{proof}
For each \(\rho\in\Irr(G)\setminus\{\mathbf{1}_G\}\), let
\(a_\rho^H\) denote the multiplicity of \(\rho\) in
\(\Ind_H^G(\mathbf{1}_H)\), as in \cref{lem:main}.
By Frobenius reciprocity, the multiplicity \(a_\rho^H\) is
\[
a_\rho^H
=
\dim \Hom_G(\Ind_{H}^{G}\mathbf{1}_H,\rho)
=
\dim \Hom_H(\mathbf{1}_H,\Res_H^G\rho)
=
\dim V_\rho^H.
\]
Since \(H\) is normal in \(G\), the subspace \(V_\rho^H\) is
\(G\)-stable. Hence, by the irreducibility of \(\rho\),
\[
V_\rho^H=
\begin{cases}
V_\rho,& H\subseteq\Ker\rho,\\
0,& H\not\subseteq\Ker\rho.
\end{cases}
\]
Thus
\[
a_\rho^H=
\begin{cases}
d_\rho,& H\subseteq\Ker\rho,\\
0,& H\not\subseteq\Ker\rho.
\end{cases}
\]
The assertion follows from \cref{lem:main}.
\end{proof}
\section{Brauer--Kuroda relations for graphs}
\subsection{The M\"{o}bius functions of finite posets}\label{mobi fu}
		
We recall some basic facts about Möbius functions of finite partially ordered sets.
For more details, we refer to \cite{Sta12} and \cite{Zas87}.


\begin{dfn}\label{dfn:mob}
Let \(\scL\) be a finite partially ordered set. The \tit{M\"{o}bius function} of \(\scL\) is the function 
\(\mu\colon \scL \times \scL \to \bbZ\) satisfying the following condition 
\[
\mu(x,y)=\begin{cases}
	\hfill 1 \hfill & \text{if \(x=y\),}\\
	\displaystyle{-\sum_{\substack{z\in \scL \\ x\leq z< y}}\mu(x,z)} & \text{if \(x<y\),}\\
	\hfill 0 \hfill & \text{otherwise}.
\end{cases} 
\]
\end{dfn}
These conditions can be equivalently written as (see Proposition 7.1.2. in \cite{Zas87}): 
\[
\mu(x,y)=\begin{cases}
	\hfill 1 \hfill & \text{if \(x=y\),}\\
	\displaystyle{-\sum_{\substack{z\in \scL \\ x< z\leq y}}\mu(z,y)} & \text{if \(x<y\),}\\
	\hfill 0 \hfill  & \text{otherwise}.
\end{cases}
\]
\begin{prp}[M\"{o}bius Inversion Formula, {\cite[Proposition 3.7.2]{Sta12}}]\label{Moinv}
Let \(\scL\) be a finite partially ordered set. Let \(f,g\colon\scL\to A\) be two maps, where \(A\) is an abelian group. Then we have
\[
\text{\(g(x)=\sum_{\substack{y\in \scL\\y\geq x}}f(y)\) if and only if \(f(x)=\sum_{\substack{y\in \scL\\y\geq x}}\mu(x,y) g(y)\)}
\]
\end{prp}
\subsection{A Kuroda-type formula}
In this subsection, we prove \cref{main3} and examine several examples.
\begin{dfn}\label{dfn:h}
For a finite group \(G\), we define 
\[
\calH_{G}\coloneq \{\, H\mid H=\Ker(\rho) \text{ for some }\rho\in \Irr(G)\,\}.
\] 
We also define
\[
\underline{\calH}_{G}\coloneq \calH_{G}\sqcup\{\emptyset\},
\]
where \(\emptyset\) denotes the empty set. 
The set \(\underline{\calH}_{G}\) is partially ordered by inclusion.
\end{dfn}
\begin{thm}\label{main3}
Let \(G\) be a finite group, and let \(\underline{\calH}_{G}\) be the partially ordered set defined in \cref{dfn:h}. Let
\(
\mu\colon \underline{\calH}_{G}\times \underline{\calH}_{G}\to \bbZ
\)
be the M\"obius function of \(\underline{\calH}_{G}\) (cf.~\cref{dfn:mob}). 
Then, for any \(G\)-cover \(Y/X\), we have
\[
\kappa(Y)
=
\frac{1}{|G|}
\left(\prod_{v\in V_X}m_v(Y/X)\right)
\prod_{H\in \calH_G}
\left(
\frac{[G:H]\kappa(X_H)}{f_{H}(Y/X)}
\right)^{-\mu(\emptyset,H)}.
\]
\end{thm}
\begin{proof}
Define two functions \(f,g\colon \underline{\calH}_{G}\to \bbC^{*}\) by 
\[
f(H)\coloneq
\begin{cases}
\hfill \kappa(X) \hfill & \text{if \(H=G,\)}\\
\displaystyle{\prod_{\substack{\rho\in \Irr(G)\setminus\{\mathbf{1}_{G}\} \\ H=\Ker\rho}}h_{Y/X}(\rho,1)^{d_{\rho}}} & \text{otherwise},
\end{cases}
\]
\[
g(H)
\coloneq 
\kappa(X)
\hspace{-4mm}
\prod_{\substack{\rho\in \Irr(G)\setminus\{\mathbf{1}_{G}\} \\ H\subseteq \Ker\rho}}
h_{Y/X}(\rho,1)^{d_{\rho}}.
\]
By the definition of \(f\) and \(g\), for every \(H\in\underline{\calH}_G\), we have
\[
g(H)
=
\prod_{\substack{K\in \underline{\calH}_{G} \\ H\subseteq K}}
f(K).
\]
By the Möbius inversion formula (see \cref{Moinv}), we have
\[
f(H)
=
\prod_{\substack{K\in \underline{\calH}_{G} \\ H\subseteq K}}
g(K)^{\mu(H,K)}.
\]
Since \(f(\emptyset)=1\), we have 
\begin{equation}\label{eq:mob}
\prod_{K\in \underline{\calH}_{G}}
g(K)^{\mu(\emptyset,K)}
=
1.
\end{equation}
Hence we obtain
\begin{align}\label{eq:5.2}
\frac{|G|\kappa(Y)}{\prod_{v\in V_X}m_{v}(Y/X)}
&\overset{\mathclap{\ref{main2}}}{=}\notag
\kappa(X)
\prod_{\rho\in \Irr(G)\setminus\{\mathbf{1}_{G}\}}
h_{Y/X}(\rho,1)^{d_{\rho}}\\
&=\notag
g(\emptyset)\\
&\overset{\mathclap{\eqref{eq:mob}}}{=}
\prod_{H\in \calH_{G}}
g(H)^{-\mu(\emptyset,H)}.
\end{align}
Moreover, by \cref{cor:normal-intermediate-complexity}, we have 
\begin{equation}\label{eq:5.3}
g(H)=\frac{[G:H]\kappa(X_H)}{f_{H}(Y/X)}.
\end{equation}
By \eqref{eq:5.2} and \eqref{eq:5.3}, we conclude that
\[
\kappa(Y)
=
\frac{1}{|G|}
\left(\prod_{v\in V_X}m_v(Y/X)\right)
\prod_{H\in \calH_G}
\left(
\frac{[G:H]\kappa(X_H)}{f_{H}(Y/X)}
\right)^{-\mu(\emptyset,H)},
\]
as desired.
\end{proof}

\begin{eg}\label{eg:elementary-abelian-branched}
Let \(\pi\colon Y\to X\) be a \(G\)-cover with
\(
G=(\bbZ/2\bbZ)^m
\)
with \(m\geq 2\).
The group \(G\) has \(2^m-1\) subgroups
\(
H_1,\ldots,H_{2^m-1}
\)
of index \(2\), and we have
\[
\underline{\calH}_G
=
\{\emptyset,H_1,\ldots,H_{2^m-1},G\}.
\]
The M\"obius function \(\mu\) of \(\underline{\calH}_G\) satisfies
\(
\mu(\emptyset,G)=2^m-2,
\)
and 
\(
\mu(\emptyset,H_i)=-1
\)
for all
\(
i\in \{1,\dots,2^{m}-1\}.
\)
Write
\(m_v(Y/X)=2^{r_v}\)
with \(0\leq r_v\leq m\).
Let \(X_i\) denote the intermediate graph corresponding to \(H_i\).

Since \(f_G(Y/X)=1\), applying \cref{main3} gives
\[
\kappa(Y)
=\frac{2^{2^m-m-1}}{\kappa(X)^{2^m-2}}
\frac{2^{\sum_{v\in V_X}r_v}}
     {\prod_{i=1}^{2^m-1}f_{H_i}(Y/X)}
\prod_{i=1}^{2^m-1}\kappa(X_i).
\]
For each \(v\in V_X\), choose a vertex \(w_v\) above \(v\) and put
\(I_v\coloneq\Stab_G(w_v)\), so that \(|I_v|=2^{r_v}\).
Since \(G\) is abelian, \(I_v\) is independent of the choice of
\(w_v\). Let \(\pi_i\colon X_i\to X\) be the natural map.

If \(I_v\subseteq H_i\), then \(\pi_i^{-1}(v)\) consists of two
vertices, each with ramification index \(2^{r_v}\) in \(Y/X_i\).
If \(I_{v}\nsubseteq H_{i}\), \(I_vH_i=G\) and \(|I_v\cap H_i|=2^{r_v-1}\), so
\(\pi_i^{-1}(v)\) consists of one vertex, whose ramification index
in \(Y/X_i\) is \(2^{r_v-1}\). Thus
\[
\prod_{x\in\pi_i^{-1}(v)}m_x(Y/X_i)
=\begin{cases}
2^{2r_v},& I_v\subseteq H_i,\\
2^{r_v-1},& I_v\not\subseteq H_i.
\end{cases}
\]
By the definition of \(f_{H_i}(Y/X)\), it follows that
\[
f_{H_i}(Y/X)
=\prod_{\substack{v\in V_X\\I_v\subseteq H_i}}2^{-r_v}
 \prod_{\substack{v\in V_X\\I_v\not\subseteq H_i}}2.
\]

The subgroups of index \(2\) in \(G\) containing \(I_v\) correspond
to the subgroups of index \(2\) in
\(G/I_v\cong(\bbZ/2\bbZ)^{m-r_v}\).
Hence exactly \(2^{m-r_v}-1\) of the subgroups \(H_i\) contain
\(I_v\), while exactly \(2^m-2^{m-r_v}\) of them do not contain \(I_{v}\).
Consequently,
\begin{align*}
\prod_{i=1}^{2^m-1}f_{H_i}(Y/X)
&=2^{\sum_{v\in V_X}
\bigl(-r_v(2^{m-r_v}-1)+2^m-2^{m-r_v}\bigr)}\\
&=2^{\sum_{v\in V_X}
\bigl(r_v+2^m-(r_v+1)2^{m-r_v}\bigr)}.
\end{align*}
Substituting this expression into the preceding formula for
\(\kappa(Y)\), we obtain
\begin{equation*}\tag{\(\varheartsuit\)}
\kappa(Y)
=\frac{2^{2^m-m-1+
\sum_{v\in V_X}\bigl((r_v+1)2^{m-r_v}-2^m\bigr)}}
{\kappa(X)^{2^m-2}}
\prod_{i=1}^{2^m-1}\kappa(X_i).
\end{equation*}

In particular, if \(Y/X\) is unramified, then \(r_v=0\) for every
\(v\in V_X\), and hence
\(
(r_v+1)2^{m-r_v}-2^m=0
\)
for every \(v\in V_X\).
Thus \((\varheartsuit)\) reduces to
\begin{equation}\label{eq:HMSV}
\kappa(Y)
=
\frac{2^{2^m-m-1}}
{\kappa(X)^{2^m-2}}
\prod_{i=1}^{2^m-1}\kappa(X_i).
\end{equation}
\eqref{eq:HMSV} recovers the unramified formula of Hammer, Mattman, Sands, and Valli\`eres \cite[Theorem~3.6 and Remark~3.7]{HMSV24},
and see also \cite[Example~4.5]{Miz26}.
Therefore, \((\varheartsuit)\) may be regarded as a ramified generalization of \eqref{eq:HMSV}.
\end{eg}
\begin{eg}[{cf.~\cite[Example~4.6]{Miz26}}]
\label{eg:Z2-Z6-explicit}
Let
\(
G=\bbZ/2\bbZ\times\bbZ/6\bbZ.
\)
Let \(X=K_3\) be the complete graph on three
vertices, with
\(
V_X=\{v_0,v_1,v_2\}.
\)
Choose an orientation
\(
S=\{e_0,e_1,e_2\}
\)
of \(X\) such that
\((o(e_0),t(e_0))=(v_0,v_1)\),
\((o(e_1),t(e_1))=(v_1,v_2)\),
and
\((o(e_2),t(e_2))=(v_2,v_0)\).
Define a voltage assignment by
\[
\alpha(e_0)=\alpha(e_1)=\alpha(e_2)=(0,0)
\]
and put
\(I_{v_0}=\langle(1,0)\rangle\),
\(I_{v_1}=\langle(0,3)\rangle\),
and
\(I_{v_2}=\langle(0,2)\rangle\).
Since
\(
\langle I_{v_0},I_{v_1},I_{v_2}\rangle
=
G,
\)
the derived graph \(Y=X(\alpha,\calI)\) is connected by
\cref{prp:connectedness}. Thus \(Y/X\) is a \(G\)-cover.
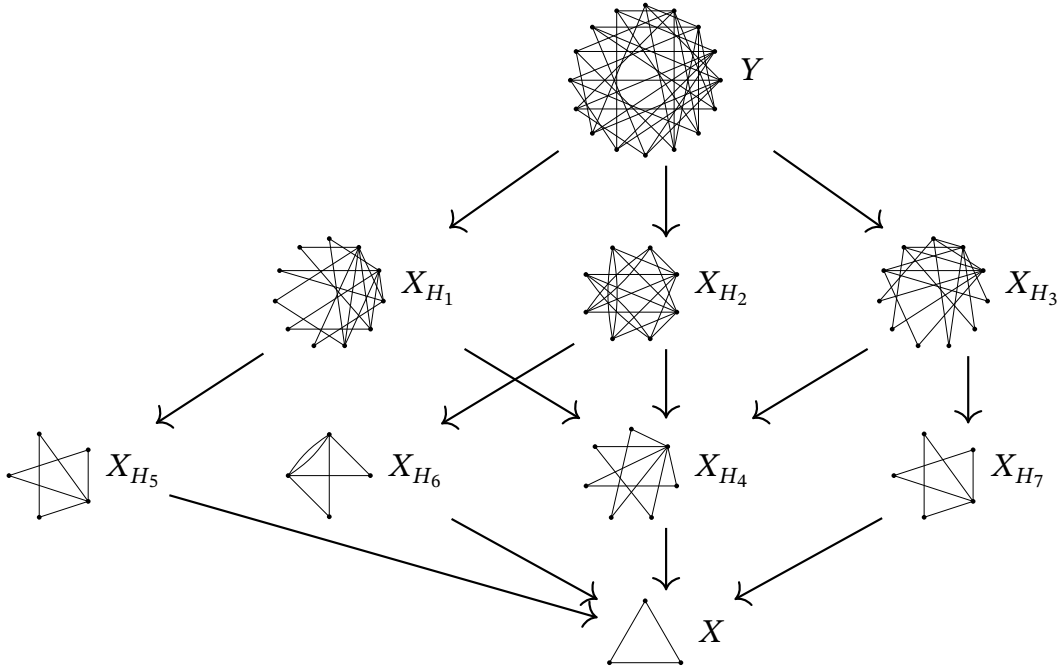
\begin{figure}[htbp]
\centering
\begin{tikzcd}[
  column sep=3em,
  row sep=2em,
  font=\small,
  scale cd=0.6,
  nodes in empty cells
]
{}
&
{}
&
\begin{tikzpicture}
\node[
  draw=none,
  minimum size=5.5cm,
  regular polygon,
  regular polygon sides=16,
  rotate=11.25
] (a) {};

\foreach \x in {1,2,...,16}{
  \pgfmathtruncatemacro{\xx}{\x}
  \fill (a.corner \xx) circle[radius=1.5pt];
}

\foreach \y/\z in {
  1/7,2/8,3/9,4/7,5/8,6/9,
  1/10,2/11,3/12,4/10,5/11,6/12,
  7/13,8/14,9/13,7/14,8/13,9/14,
  10/15,11/16,12/15,10/16,11/15,12/16,
  1/13,2/14,3/13,4/14,5/13,6/14,
  1/15,2/16,3/15,4/16,5/15,6/16
}{
  \pgfmathtruncatemacro{\yy}{\y}
  \pgfmathtruncatemacro{\zz}{\z}
  \path (a.corner \yy) edge (a.corner \zz);
}
\end{tikzpicture}
\quad \text{\LARGE\(Y\)} \arrow[d]  \arrow[ld] \arrow[rd]  
& {}  \\  
& 
\begin{tikzpicture}
\node[
  draw=none,
  minimum size=4cm,
  regular polygon,
  regular polygon sides=11,
  rotate=0
] (a) {};

\foreach \x in {1,2,...,11}{
  \pgfmathtruncatemacro{\xx}{\x}
  \fill (a.corner \xx) circle[radius=1.5pt];
}

\foreach \y/\z in {
  1/7,7/10,1/10,
  2/8,8/11,2/11,
  3/9,9/10,3/10,
  4/7,7/11,4/11,
  5/8,8/10,5/10,
  6/9,9/11,6/11
}{
  \pgfmathtruncatemacro{\yy}{\y}
  \pgfmathtruncatemacro{\zz}{\z}
  \path (a.corner \yy) edge (a.corner \zz);
}
\end{tikzpicture}
\quad \text{\LARGE\(X_{H_{1}}\)}\arrow[ld]\arrow[rd]   
& 
\begin{tikzpicture}
\node[
  draw=none,
  minimum size=3.6cm,
  regular polygon,
  regular polygon sides=8,
  rotate=0
] (b) {};

\foreach \x in {1,2,...,8}{
  \pgfmathtruncatemacro{\xx}{\x}
  \fill (b.corner \xx) circle[radius=1.5pt];
}

\foreach \y/\z in {
  4/7,1/7,
  5/8,2/8,
  6/7,3/7,
  4/8,1/8,
  5/7,2/7,
  6/8,3/8
}{
  \pgfmathtruncatemacro{\yy}{\y}
  \pgfmathtruncatemacro{\zz}{\z}
  \path (b.corner \yy) edge (b.corner \zz);
}

\foreach \y/\z in {1/4,2/5,3/6}{
  \pgfmathtruncatemacro{\yy}{\y}
  \pgfmathtruncatemacro{\zz}{\z}
  \path (b.corner \yy) edge[bend left=10] (b.corner \zz);
  \path (b.corner \yy) edge[bend right=10] (b.corner \zz);
}
\end{tikzpicture}
\quad \text{\LARGE\(X_{H_{2}}\)}\arrow[ld]\arrow[d]  
& 
\begin{tikzpicture}
\node[
  draw=none,
  minimum size=4cm,
  regular polygon,
  regular polygon sides=11,
  rotate=0
] (a) {};

\foreach \x in {1,2,...,11}{
  \pgfmathtruncatemacro{\xx}{\x}
  \fill (a.corner \xx) circle[radius=1.5pt];
}

\foreach \y/\z in {
  1/4,4/10,1/10,
  2/5,5/10,2/10,
  3/6,6/10,3/10,
  1/7,7/11,1/11,
  2/8,8/11,2/11,
  3/9,9/11,3/11
}{
  \pgfmathtruncatemacro{\yy}{\y}
  \pgfmathtruncatemacro{\zz}{\z}
  \path (a.corner \yy) edge (a.corner \zz);
}
\end{tikzpicture}
\quad \text{\LARGE\(X_{H_{3}}\)}\arrow[ld]\arrow[d] \\  
\begin{tikzpicture}
\node[
  draw=none,
  minimum size=3.2cm,
  regular polygon,
  regular polygon sides=5,
  rotate=90
] (a) {};

\foreach \x in {1,2,...,5}{
  \pgfmathtruncatemacro{\xx}{\x}
  \fill (a.corner \xx) circle[radius=1.5pt];
}

\foreach \y/\z in {
  1/3,3/4,1/4,
  2/3,3/5,2/5
}{
  \pgfmathtruncatemacro{\yy}{\y}
  \pgfmathtruncatemacro{\zz}{\z}
  \path (a.corner \yy) edge (a.corner \zz);
}
\end{tikzpicture}
\quad \text{\LARGE\(X_{H_{5}}\)}\arrow[rrd]  
&
\begin{tikzpicture}
\node[
  draw=none,
  minimum size=3cm,
  regular polygon,
  regular polygon sides=4,
  rotate=45
] (b) {};

\foreach \x in {1,2,...,4}{
  \pgfmathtruncatemacro{\xx}{\x}
  \fill (b.corner \xx) circle[radius=1.5pt];
}

\foreach \y/\z in {
  2/3,1/3,
  2/4,1/4
}{
  \pgfmathtruncatemacro{\yy}{\y}
  \pgfmathtruncatemacro{\zz}{\z}
  \path (b.corner \yy) edge (b.corner \zz);
}

\path (b.corner 1) edge[bend left=12] (b.corner 2);
\path (b.corner 1) edge[bend right=12] (b.corner 2);
\end{tikzpicture}
\quad \text{\LARGE\(X_{H_{6}}\)}\arrow[rd] 
&
\begin{tikzpicture}
\node[
  draw=none,
  minimum size=3.4cm,
  regular polygon,
  regular polygon sides=7,
  rotate=0
] (b) {};

\foreach \x in {1,2,...,7}{
  \pgfmathtruncatemacro{\xx}{\x}
  \fill (b.corner \xx) circle[radius=1.5pt];
}

\foreach \y/\z in {
  1/4,4/7,1/7,
  2/5,5/7,2/7,
  3/6,6/7,3/7
}{
  \pgfmathtruncatemacro{\yy}{\y}
  \pgfmathtruncatemacro{\zz}{\z}
  \path (b.corner \yy) edge (b.corner \zz);
}
\end{tikzpicture}
\quad \text{\LARGE\(X_{H_{4}}\)}\arrow[d]
&
\begin{tikzpicture}
\node[
  draw=none,
  minimum size=3.2cm,
  regular polygon,
  regular polygon sides=5,
  rotate=90
] (a) {};

\foreach \x in {1,2,...,5}{
  \pgfmathtruncatemacro{\xx}{\x}
  \fill (a.corner \xx) circle[radius=1.5pt];
}

\foreach \y/\z in {
  1/3,3/4,1/4,
  2/3,3/5,2/5
}{
  \pgfmathtruncatemacro{\yy}{\y}
  \pgfmathtruncatemacro{\zz}{\z}
  \path (a.corner \yy) edge (a.corner \zz);
}
\end{tikzpicture}
\quad \text{\LARGE\(X_{H_{7}}\)}\arrow[ld]\\ 
 {}  &  {}  & 
\begin{tikzpicture}
\node[
  draw=none,
  minimum size=3cm,
  regular polygon,
  regular polygon sides=3,
  rotate=0
] (b) {};

\foreach \x in {1,2,3}{
  \pgfmathtruncatemacro{\xx}{\x}
  \fill (b.corner \xx) circle[radius=1.5pt];
}

\foreach \y/\z in {
  1/2,2/3,3/1
}{
  \pgfmathtruncatemacro{\yy}{\y}
  \pgfmathtruncatemacro{\zz}{\z}
  \path (b.corner \yy) edge (b.corner \zz);
}
\end{tikzpicture}
\quad \text{\LARGE\(X\)}
&  {}     	
\end{tikzcd}
\caption{The quotient graphs associated with the subgroups of \(G\)}
\label{fig}
\end{figure}
We have
\[
\calH_G=\{H_1,\dots,H_8\},
\]
where
\(H_1=\langle(1,0)\rangle\),
\(H_2=\langle(1,3)\rangle\),
\(H_3=\langle(0,3)\rangle\),
\(H_4=\langle(1,0),(0,3)\rangle\),
\(H_5=\langle(1,2)\rangle\),
\(H_6=\langle(1,1)\rangle\),
\(H_7=\langle(0,1)\rangle\)
and
\(H_8=G\).
The corresponding intermediate graphs are shown in Figure \ref{fig}.
Moreover,
\[
\mu(\emptyset,H_i)=
\begin{cases}
-1 & \text{if }i\in\{1,2,3\},\\
2  & \text{if }i=4,\\
0  & \text{if }i\in\{5,6,7,8\}.
\end{cases}
\]
Using SageMath \cite{Sage}, we compute
\[
\kappa(X_{H_{1}})=3888,\quad \kappa(X_{H_{2}})=5184,\quad \kappa(X_{H_{3}})=3888,\quad \kappa(X_{H_{4}})=27.
\]
Moreover, we can compute 
\(f_{H_{1}}(Y/X)=3/16\),
\(f_{H_{2}}(Y/X)=12\),
\(f_{H_{3}}(Y/X)=3/16\),
and
\(f_{H_{4}}(Y/X)=3/16\).
By \cref{main3}, we have
\[
\kappa(Y)
=
\frac{1}{|G|}
\left(
\prod_{v\in V_{X}}m_{v}(Y/X)
\right)
\frac{
\displaystyle
\prod_{i=1}^{3}
\left(
\dfrac{[G:H_i]\kappa(X_{H_i})}
{f_{H_i}(Y/X)}
\right)
}{
\left(
\dfrac{[G:H_4]\kappa(X_{H_4})}
{f_{H_4}(Y/X)}
\right)^2
}\\
=
214990848.
\]
A computation of the number of spanning trees of \(Y\) using SageMath \cite{Sage} gives
\(
\kappa(Y)=214990848,
\)
which agrees with the above calculation.
\end{eg}

\subsection{Brauer--Kuroda relations}
In this subsection, we prove \cref{main4} and consider several examples. 
\begin{dfn}\label{dfn:c}
For a finite group \(G\), we define 
\[
\calC_{G}\coloneq \{\, C\mid \text{\(C\) is a cyclic subgroup of \(G\)}\, \}.
\]
\end{dfn}

\begin{thm}[Artin's induction theorem, cf.~{\cite[Theorem 2.1.3]{Sna94}}]\label{artinn}
Let $\chi$ be a rational-valued character of $G$. For every cyclic subgroup $C$ of $G$, let 
\[
a_{\chi}(C)\coloneq \frac{1}{[G:C]}\sum_{\substack{B\in \calC_{G} \\ C\subseteq B}}\mu_{\mathrm{cl}}([B:C])\chi(g_B),
\] 
where $g_{B}$ is a generator of $B$ and $\mu_{\mathrm{cl}}\colon \bbN\setminus\{0\}\to \{-1,0,1\}$ is the classical M\"{o}bius function. Then we have $$\chi=\sum_{\substack{C\in \calC_{G}}}a_{\chi}(C)\chi_{\Ind^{G}_{C}(\mathbf{1}_{C})}. $$ 
\end{thm}
\begin{cor}\label{cor:c}
Let \(\chi_0\) be the trivial character of \(G\). For each
\(C\in\calC_G\), define
\[
a(C)
\coloneq
\frac{1}{[G:C]}
\sum_{\substack{B\in\calC_G\\ C\subseteq B}}
\mu_{\mathrm{cl}}([B:C]).
\]
Then we have
\[
\chi_0
=
\sum_{C\in\calC_G}
a(C)\,
\chi_{\Ind_C^G(\mathbf{1}_{C})}.
\]
\end{cor}
\begin{proof}
Apply \cref{artinn} to the trivial character \(\chi_0\).
\end{proof}

\begin{thm}\label{main4}
Let \(G\) be a finite group, and let \(\calC_{G}\) be the set defined in \cref{dfn:c}. 
Then, for any \(G\)-cover \(Y/X\), we have
\[
\kappa(X)
=
\prod_{C\in \calC_{G}}
\left(
\frac{[G:C]\kappa(X_{C})}
{f_{C}(Y/X)}
\right)^{a(C)}.
\]
\end{thm}
\begin{proof}
Let \(a^{C}_{\rho}\) be the integers satisfying 
\[
\chi_{\Ind^{G}_{C}\mathbf{1}_{C}}=\chi_{\mathbf{1}_{G}}+\sum_{\rho\in \Irr(G)\setminus \{\mathbf{1}_{G}\}}a_{\rho}^{C}\chi_{\rho}
\] 
for each \(C\in \calC_{G}\).
By \cref{cor:c}, we have 
\begin{align*}
\chi_{\mathbf{1}_{G}}
&\overset{\mathclap{\ref{cor:c}}}{=}
\sum_{C\in \calC_{G}}
\left(
a(C)
\left(
\chi_{\mathbf{1}_{G}}
+
\sum_{\rho\in \Irr(G)\setminus \{\mathbf{1}_{G}\}}
a^{C}_{\rho}\chi_{\rho}
\right)
\right)
\\
&=
\left(
\sum_{C\in \calC_{G}} a(C)
\right)
\chi_{\mathbf{1}_{G}}
+
\sum_{\rho\in \Irr(G)\setminus \{\mathbf{1}_{G}\}}
\left(
\sum_{C\in \calC_{G}}a(C)a^{C}_{\rho}
\right)
\chi_{\rho}.
\end{align*}
Since the irreducible characters of \(G\) are linearly independent,
we obtain
\begin{equation}\label{eq:aC}
\sum_{C\in \calC_G} a(C)=1,
\qquad
\sum_{C\in \calC_G}a(C)a^C_\rho=0
\end{equation}
for any \(\rho\in \Irr(G)\setminus\{\mathbf{1}_{G}\}\).
Then we have
\begin{align*}
&\prod_{C\in \calC_{G}}([G:C]\kappa(X_{C}))^{a(C)}\\
&\overset{\mathclap{\ref{lem:main}}}{=}
\prod_{C\in \calC_{G}}
\left(
f_{C}(Y/X)
\cdot
\kappa(X)
\hspace{-4mm}
\prod_{\rho\in\Irr(G)\setminus\{\mathbf{1}_{G}\}}
h_{Y/X}(\rho,1)^{a_\rho^C}
\right)^{a(C)}
\\
&=
\kappa(X)^{\sum_{C\in\calC_G}a(C)}
\prod_{C\in \calC_{G}}
{f_{C}(Y/X)}^{a(C)}
\prod_{\rho\in\Irr(G)\setminus\{\mathbf{1}_{G}\}}
h_{Y/X}(\rho,1)^{
\sum_{C\in\calC_G}a(C)a^C_\rho
}
\\
&\overset{\mathclap{\eqref{eq:aC}}}{=}
\kappa(X)
\prod_{C\in \calC_{G}}
{f_{C}(Y/X)}^{a(C)},
\end{align*}
as desired.
\end{proof}
\begin{rmk}
We make a few remarks on the main results.
\begin{enumerate}
\item
Let \(\ol{\calC}_G\coloneq\calC_G\sqcup\{\infty\}\) be the
partially ordered set obtained by adjoining a greatest element
\(\infty\) to \(\calC_G\), and let
\(\mu\colon\ol{\calC}_G\times\ol{\calC}_G\to\bbZ\)
be its M\"obius function.
Then, for every \(C\in\calC_G\),
\[
a(C)=-\frac{1}{[G:C]}\mu(C,\infty).
\]
Thus \cref{main4} can equivalently be written as
\[
\kappa(X)
=\prod_{C\in\calC_G}
\left(
\frac{[G:C]\kappa(X_C)}{f_C(Y/X)}
\right)^{-\frac{1}{[G:C]}\mu(C,\infty)}.
\]
This is the ramified analogue of \cite[Theorem~4.11]{Miz26}.

\item
Suppose that \(Y/X\) is unramified. Then \(m_v(Y/X)=1\) for
every \(v\in V_X\), and hence \(f_H(Y/X)=1\) for every subgroup
\(H\subseteq G\).
Therefore, \cref{main3,main4} specialize to the Kuroda and
Brauer--Kuroda relations for unramified \(G\)-covers obtained
in \cite[Theorem~4.3 and 4.11]{Miz26}.

\item
A finite group \(G\) is called \textit{irreducibly represented}
if it admits a faithful irreducible complex representation.
Every cyclic group has this property.
If \(G\) is irreducibly represented, then \cref{main3} reduces
to an identity.
Indeed, we have \(\{1_G\}\in\calH_G\).
For the M\"obius function \(\mu\) of \(\underline{\calH}_G\),
it follows that
\(\mu(\emptyset,\{1_G\})=-1\)
and
\(\mu(\emptyset,H)=0\)
for every \(H\in\calH_G\setminus\{\{1_G\}\}\).
Thus \cref{main3} reduces to the identity \(\kappa(Y)=\kappa(Y)\).
A list of groups of order at most \(24\), indicating which
are irreducibly represented, is given in
\cite[Table~A.1]{Miz26}.

\item
Suppose that \(G\) is cyclic in \cref{main4}.
Then every subgroup of \(G\) is cyclic, and
\[
a(C)=\begin{cases}
1,& \text{if \(C=G\)},\\
0,& \text{if \(C\neq G\)}.
\end{cases}
\]
Thus \cref{main4} reduces to the identity
\(\kappa(X)=\kappa(X)\).
If \(G\) is noncyclic, then \(G\notin\calC_G\), so
\(\kappa(X)\) does not appear on the right-hand side of the
formula in \cref{main4}. Hence the formula gives a nontrivial
relation in this case.

\item
A relation of the form 
\[
\sum_{H\subseteq G}n_{H}\Ind_{H}^{G}\mathbf{1}_{H}=0
\]
is called a \tit{Brauer relation}, where \(n_{H}\in \bbZ\). 
Artin's induction theorem (=\cref{artinn}) is an example of a Brauer relation.
Brauer relations in finite groups are studied and classified in \cite{BD15}.
Given a Brauer relation for \(G\), an argument similar to that in the proof of
\cref{main4} gives
\[
\prod_{H\subseteq G}\left(\frac{[G:H]\kappa(X_{H})}{f_{H}(Y/X)} \right)^{n_{H}}=1
\]
A related observation in the unramified setting appears in
\cite[Remark~4.13]{Miz26}.

\item
In \cite[Theorem~4.22]{Miz26}, we proved that no nontrivial monomial-type spanning tree formula exists for unramified \(G\)-covers when \(G\) is a cyclic group.
Thus, for cyclic groups, the triviality described in (3) and (4) reflects a nonexistence result already present in the unramified setting.
\end{enumerate}
\end{rmk}
\begin{eg}\label{eg:elementary-abelian-artin}
Let \(G=(\bbZ/2\bbZ)^m\), with \(m\geq2\), and let \(Y/X\) be a \(G\)-cover.
We show that (\(\varheartsuit\)) in \cref{eg:elementary-abelian-branched} also follows from \cref{main4}.
Let \(C_1,\ldots,C_{2^m-1}\) be the subgroups of \(G\) of order two, and let \(H_1,\ldots,H_{2^m-1}\) be its subgroups of index two.

Applying \cref{main4} to \(Y/X\) and taking the power \(2^{m-1}\) gives
\begin{equation}\label{eq:el2-direct-original}
  \kappa(X)^{2^{m-1}}
  =
  \left(
    \frac{2^m\kappa(Y)}{f_{\{1_G\}}(Y/X)}
  \right)^{1-2^{m-1}}
  \prod_{i=1}^{2^m-1}
  \frac{2^{m-1}\kappa(X_{C_i})}{f_{C_i}(Y/X)}.
\end{equation}
Applying \cref{main4} to the \(H_j\)-cover \(Y/X_{H_j}\) and 
raising both sides to the power \(2^{m-2}\), we obtain
\[
\kappa(X_{H_j})^{2^{m-2}}
=
\left(
\frac{2^{m-1}\kappa(Y)}
{f_{\{1_G\}}(Y/X_{H_j})}
\right)^{1-2^{m-2}}
\prod_{\substack{C\subseteq H_j\\ |C|=2}}
\frac{2^{m-2}\kappa(X_C)}
{f_C(Y/X_{H_j})}.
\]
By the definition of the ramification factor, we have
\[
f_{\{1_G\}}(Y/X_{H_j})
=
\frac{f_{\{1_G\}}(Y/X)}{f_{H_j}(Y/X)}.
\]
Thus, we obtain
\begin{equation}\label{eq:el2-direct-hyperplane}
\left(
\frac{2\kappa(X_{H_j})}{f_{H_j}(Y/X)}
\right)^{2^{m-2}}
=
\left(
\frac{2^m\kappa(Y)}{f_{\{1_G\}}(Y/X)}
\right)^{1-2^{m-2}}
\prod_{\substack{C\subseteq H_j\\ |C|=2}}
\frac{2^{m-1}\kappa(X_C)}{f_C(Y/X)}.
\end{equation}

Each \(C_i\) is contained in \(2^{m-1}-1\) of the
subgroups \(H_j\), since the subgroups \(H_{j}\) containing \(C_{i}\) correspond to the
index-two subgroups of \(G/C_i\). 
Thus, multiplying \eqref{eq:el2-direct-hyperplane} over \(j=1,\ldots,2^m-1\),
we obtain
\[
\prod_{j=1}^{2^m-1}
\left(
\frac{2\kappa(X_{H_j})}{f_{H_j}(Y/X)}
\right)^{2^{m-2}}
=
\left(
\frac{2^m\kappa(Y)}{f_{\{1_G\}}(Y/X)}
\right)^{(2^m-1)(1-2^{m-2})}
\prod_{i=1}^{2^m-1} \left( \frac{2^{m-1}\kappa(X_{C_i})}{f_{C_i}(Y/X)} \right)^{2^{m-1}-1}.
\]
Substituting \eqref{eq:el2-direct-original} into this equation, we obtain
\[
  \left(
    \prod_{j=1}^{2^m-1}
    \frac{2\kappa(X_{H_j})}{f_{H_j}(Y/X)}
  \right)^{2^{m-2}}
  =
  \left(
    \frac{2^m\kappa(Y)}{f_{\{1_G\}}(Y/X)}
    \kappa(X)^{2^m-2}
  \right)^{2^{m-2}}.
\]
Hence,
\begin{equation}\label{eq:el2-direct-factor-form}
  \kappa(Y)
  =
  \frac{2^{\,2^m-m-1}}{\kappa(X)^{\,2^m-2}}
  \frac{f_{\{1_G\}}(Y/X)}
       {\displaystyle\prod_{j=1}^{2^m-1}f_{H_j}(Y/X)}
  \prod_{j=1}^{2^m-1}\kappa(X_{H_j}).
\end{equation}

For each \(v\in V_X\), write \(m_v(Y/X)=2^{r_v}\).
The ramification-factor calculation in
\cref{eg:elementary-abelian-branched} gives
\[
  \frac{f_{\{1_G\}}(Y/X)}
       {\displaystyle\prod_{j=1}^{2^m-1}f_{H_j}(Y/X)}
  =
  2^{\sum_{v\in V_X}
       \left((r_v+1)2^{m-r_v}-2^m\right)}.
\]
Substituting this into \eqref{eq:el2-direct-factor-form}, we obtain
\[
  \kappa(Y)
  =
  \frac{
    2^{2^m-m-1+
       \sum_{v\in V_X}\left((r_v+1)2^{m-r_v}-2^m\right)}
  }{\kappa(X)^{\,2^m-2}}
  \prod_{j=1}^{2^m-1}\kappa(X_{H_j}),
\]
which is \((\varheartsuit)\).
\end{eg}
\begin{eg}
\begin{figure}[t]
\centering
\begin{tikzcd}[
  column sep=1.5em,
  row sep=1em,
  font=\small,
  scale cd=0.6,
  nodes in empty cells
]
{}
&
{}
&
\begin{tikzpicture}
\node[
  draw=none,
  minimum size=4.5cm,
  regular polygon,
  regular polygon sides=11,
  rotate=48
] (a) {};

\foreach \x in {1,2,...,11}{
  \pgfmathtruncatemacro{\xx}{\x}
  \fill (a.corner \xx) circle[radius=1.5pt];
}

\foreach \y/\z in {
  1/3,1/2,2/3,
  1/7,1/11,11/7,
  5/3,5/4,4/3,
  5/7,5/6,6/7,
  9/3,9/10,10/3,
  9/7,9/8,8/7
}{
  \pgfmathtruncatemacro{\yy}{\y}
  \pgfmathtruncatemacro{\zz}{\z}
  \path (a.corner \yy) edge (a.corner \zz);
}
\end{tikzpicture}
\quad \text{\LARGE\(Y\)}
\arrow[dll]
\arrow[dl]
\arrow[d]
\arrow[drr]
&
{}
&
{}
\\[2em]
\begin{tikzpicture}
\node[
  draw=none,
  minimum size=3.2cm,
  regular polygon,
  regular polygon sides=6,
  rotate=0
] (a) {};

\foreach \x in {1,2,...,6}{
  \pgfmathtruncatemacro{\xx}{\x}
  \fill (a.corner \xx) circle[radius=1.5pt];
}

\foreach \y/\z in {
  1/6,
  6/2,6/3,6/5,
  2/1,
  3/4,
  5/4
}{
  \pgfmathtruncatemacro{\yy}{\y}
  \pgfmathtruncatemacro{\zz}{\z}
  \path (a.corner \yy) edge (a.corner \zz);
}

\path (a.corner 4) edge[bend left=10] (a.corner 6);
\path (a.corner 4) edge[bend right=10] (a.corner 6);
\end{tikzpicture}
\quad \text{\Large\(X_{\langle s\rangle}\)}
\arrow[drr]
&
\begin{tikzpicture}
\node[
  draw=none,
  minimum size=3.2cm,
  regular polygon,
  regular polygon sides=6,
  rotate=0
] (a) {};

\foreach \x in {1,2,...,6}{
  \pgfmathtruncatemacro{\xx}{\x}
  \fill (a.corner \xx) circle[radius=1.5pt];
}

\foreach \y/\z in {
  1/6,
  6/2,6/3,6/5,
  2/1,
  3/4,
  5/4
}{
  \pgfmathtruncatemacro{\yy}{\y}
  \pgfmathtruncatemacro{\zz}{\z}
  \path (a.corner \yy) edge (a.corner \zz);
}

\path (a.corner 4) edge[bend left=10] (a.corner 6);
\path (a.corner 4) edge[bend right=10] (a.corner 6);
\end{tikzpicture}
\quad \text{\Large\(X_{\langle rs\rangle}\)}
\arrow[dr]
&
\begin{tikzpicture}
\node[
  draw=none,
  minimum size=3.2cm,
  regular polygon,
  regular polygon sides=6,
  rotate=0
] (a) {};

\foreach \x in {1,2,...,6}{
  \pgfmathtruncatemacro{\xx}{\x}
  \fill (a.corner \xx) circle[radius=1.5pt];
}

\foreach \y/\z in {
  1/6,
  6/2,6/3,6/5,
  2/1,
  3/4,
  5/4
}{
  \pgfmathtruncatemacro{\yy}{\y}
  \pgfmathtruncatemacro{\zz}{\z}
  \path (a.corner \yy) edge (a.corner \zz);
}

\path (a.corner 4) edge[bend left=10] (a.corner 6);
\path (a.corner 4) edge[bend right=10] (a.corner 6);
\end{tikzpicture}
\quad \text{\Large\(X_{\langle r^{2}s\rangle}\)}
\arrow[d]
&
{}
&
\begin{tikzpicture}
\node[
  draw=none,
  minimum size=3.2cm,
  regular polygon,
  regular polygon sides=5,
  rotate=90
] (a) {};

\foreach \x in {1,2,...,5}{
  \pgfmathtruncatemacro{\xx}{\x}
  \fill (a.corner \xx) circle[radius=1.5pt];
}

\foreach \y/\z in {
  1/3,3/4,1/4,
  2/3,3/5,2/5
}{
  \pgfmathtruncatemacro{\yy}{\y}
  \pgfmathtruncatemacro{\zz}{\z}
  \path (a.corner \yy) edge (a.corner \zz);
}
\end{tikzpicture}
\quad \text{\Large\(X_{\langle r\rangle}\)}
\arrow[dll]
\\[2em]
{}
&
{}
&
\begin{tikzpicture}
\node[
  draw=none,
  minimum size=2.8cm,
  regular polygon,
  regular polygon sides=3,
  rotate=0
] (a) {};

\foreach \x in {1,2,3}{
  \pgfmathtruncatemacro{\xx}{\x}
  \fill (a.corner \xx) circle[radius=1.5pt];
}

\foreach \y/\z in {
  1/2,2/3,3/1
}{
  \pgfmathtruncatemacro{\yy}{\y}
  \pgfmathtruncatemacro{\zz}{\z}
  \path (a.corner \yy) edge (a.corner \zz);
}
\end{tikzpicture}
\quad \text{\LARGE\(X\)}
&
{}
&
{}
\end{tikzcd}
\caption{The quotient graphs associated with the subgroups of \(G\)}
\label{fig:intermediate-K3-S3}
\end{figure}

We use the same setup as in \cref{eg:S3}. The cyclic subgroups of \(G=S_{3}\) are
\(
\calC_G
=
\bigl\{
\{1_G\},
\langle s\rangle,
\langle rs\rangle,
\langle r^2s\rangle,
\langle r\rangle
\bigr\}.
\)
The corresponding intermediate graphs are shown in Figure \ref{fig:intermediate-K3-S3}.
By the definition of \(a(C)\), we have
\(
a(\{1_G\})=-1/2,
a(\langle s\rangle)=
a(\langle rs\rangle)=
a(\langle r^2s\rangle)
=1/3,
a(\langle r\rangle)=1/2.
\)
Moreover, we can compute 
\(
f_{\{1_G\}}(Y/X)=6,
f_{\langle s\rangle}(Y/X)=
f_{\langle rs\rangle}(Y/X)=
f_{\langle r^2s\rangle}(Y/X)=3,
f_{\langle r\rangle}(Y/X)=2/3.
\)

Therefore, by \cref{main4},
\begin{equation*}
\kappa(Y)
=
3
\cdot
\frac{\kappa(X_{\langle s\rangle})^{2/3}
\kappa(X_{\langle rs\rangle})^{2/3}
\kappa(X_{\langle r^2s\rangle})^{2/3}
\kappa(X_{\langle r\rangle})}{\kappa(X)^2}.
\end{equation*}

Using SageMath \cite{Sage}, we compute
\(
\kappa(X_{\langle s\rangle})=
\kappa(X_{\langle rs\rangle})=
\kappa(X_{\langle r^2s\rangle})=36,
\kappa(X_{\langle r\rangle})=9.
\)
Hence, from the equation above, we obtain \(\kappa(Y)=3888\).
This agrees with the result obtained by a direct computation in SageMath \cite{Sage}.
\end{eg}


\end{document}